\documentclass[11pt]{article}

\usepackage{color,xcolor}
\usepackage[margin = 1.5in]{geometry}
\usepackage{setspace}
\usepackage[T1]{fontenc}
\usepackage{charter} 
\usepackage{euler} 
\usepackage[normalem]{ulem}

\usepackage{float}
\usepackage{soul} 
\usepackage{bbm} 
\usepackage{verbatim}
\usepackage{amsmath}
\usepackage{amsthm}
\usepackage{amssymb}
\usepackage{textcomp}
\usepackage{graphicx}
\usepackage[english]{babel}
\usepackage{tikz}
\usepackage[linguistics]{forest}
\usepackage{dblfloatfix}
\usepackage{csquotes}
\usepackage{hyperref}
\usepackage{nameref}
\usepackage{tabularray}
\usepackage{theoremref}
\usepackage{natbib}
\usepackage{comment}
\usepackage{etoolbox}
\usepackage{algorithm2e}
\usepackage{algpseudocode}
\allowdisplaybreaks 

\usepackage{enumerate}

\newtheorem{theorem}{Theorem}[section]

\newtheorem{lemma}[theorem]{Lemma}
\newtheorem*{lemma*}{Lemma}
\newtheorem*{claim}{Claim}
\newtheorem{corollary}[theorem]{Corollary}

\newtheorem*{theorem*}{Theorem}

\usepackage{titlesec}
\titleformat{\section}{\sc\Large}{\thesection}{1em}{}
\titleformat{\subsection}{\sc\large}{\thesubsection}{1em}{}

\newcommand{\E}{\mathbb{E}}

\newcommand{\N}{\mathbb{N}} 

\newcommand{\R}{\mathbb{R}} 

\newcommand{\Z}{\mathbb{Z}} 

\newcommand{\1}{\mathbf{1}}
\newcommand{\prob}{\mathbf{P}}
\newcommand{\ex}{\mathbf{E}}

\newcommand{\bin}{\ensuremath{\operatorname{Bin}}}
\newcommand{\negbin}{\ensuremath{\operatorname{N-Bin}}}
    \newcommand{\geo}{\ensuremath{\operatorname{Geo}}}
    \newcommand{\hyper}{\ensuremath{\operatorname{HG}}}
\newcommand{\ber}{\ensuremath{\operatorname{Ber}}}
\newcommand{\unif}{\ensuremath{\operatorname{Unif}}}

\newcommand{\dirichlet}{\ensuremath{\operatorname{Dirichlet}}}

\newcommand{\convdist}{\xrightarrow[]{d}}
\newcommand{\convprob}{\ensuremath{\xrightarrow[]{\mathbb{P}}}} 
\newcommand{\convas}{\ensuremath{\xrightarrow[]{\text{a.s.}}}} 
 
\newcommand{\dist}{\ensuremath{\overset{d}{=}}}

\newcommand\cE{\mathcal E}
\newcommand\cF{\mathcal F}
\newcommand\cG{\mathcal G}

\newcommand\cR{{\mathcal R}}
\newcommand\cS{{\mathcal S}}

\newcommand{\erp}{\ensuremath{\operatorname{ERP}}} 
\newcommand{\king}{\ensuremath{\textsc{Kingman}}} 

\newcommand{\height}{\ensuremath{\operatorname{height}}}
\newcommand{\urrf}{\operatorname{URRF}}

\newcommand\marginal[1]{\marginpar{\raggedright\parindent=0pt\tiny #1}}

\definecolor{clou}{rgb}{0.8,0.25,0.5125}

\newcommand{\lmar}[1]{\textcolor{clou}{\marginal{#1 \\ -lab.}}} 

\definecolor{cmax}{rgb}{0.5, 0.5, 0.8}
\newcommand{\ma}[1]{\textcolor{cmax}{#1}}
\newcommand{\mmar}[1]{\marginal{\color{cmax} #1 \\ -mk.}}

\newcommand{\cmar}[1]{\marginal{#1 \\ -ca.}}
\definecolor{cae}{rgb}{0.5,0,0.5}

\title{\sc{Kingman's coalescent on a random graph}}
\author{\sc{Louigi Addario-Berry, Caelan Atamanchuk, and Maxwell Kaye}}
\date{}

\AtEndDocument{%
  \par
  \medskip
  \begin{tabular}{@{}l@{}}%
    \textsc{Louigi Addario-Berry}\\
    \textsc{email: }\texttt{louigi@gmail.com}\\
    \textsc{Department of Mathematics and Statistics}\\
    \textsc{McGill University}\\
    \\
    \textsc{Caelan Atamanchuk}\\
    \textsc{email: }\texttt{caelan.atamanchuk@gmail.com}\\
    \textsc{Department of Mathematics and Statistics}\\
    \textsc{McGill University}\\
    \\
    \textsc{Maxwell Kaye}\\
    \textsc{email: }\texttt{maxwell.kaye@mail.mcgill.ca}\\
    \textsc{Department of Mathematics and Statistics}\\
    \textsc{McGill University}\\
  \end{tabular}}

\begin{document}

\maketitle

\setstretch{1.0}

\begin{abstract}
    We introduce a generalization of Kingman's coalescent on $[n]$ that we call the \emph{Kingman coalescent} on a graph $G = ([n],E)$. Specifically, we generalize a forest valued representation of the coalescent introduced in \cite{addario2018high}. The difference between the Kingman coalescent on $G$ and the normal Kingman coalescent on $[n]$ is that two trees $T_1,T_2$ with roots $\rho_1,\rho_2$ can merge if and only if $\{\rho_1,\rho_2\} \in E$.
    When this process finishes (when there are no trees left that can merge anymore), we are left with a random spanning forest that we call a \emph{Kingman forest} of $G$. In this article, we study the Kingman coalescent on Erd\H{o}s-R\'{e}nyi random graphs, $G_{n,p}$. We derive a relationship between the Kingman coalescent on $G_{n,p}$ and uniform random recursive trees, which provides many answers concerning structural questions about the corresponding Kingman forests. We explore the heights of Kingman forests as well as the sizes of their trees as illustrative examples of how to use the connection. Our main results concern the number of trees, $C_{n,p}$, in a Kingman forest of $G_{n,p}$. For fixed $p \in (0,1)$, we prove that $C_{n,p}$ converges in distribution to an almost surely finite random variable as $n \to \infty$. For $p = p(n)$ such that $p \to 0$ and $np \to \infty$ as $n \to \infty$, we prove that $C_{n,p}$ converges in probability to $\frac{2(1-p)}{p}$.
\end{abstract}

\section{Introduction}\label{sec:intro}

\subsection{Definitions and results}

Let $G$ be a finite graph with $|V(G)|=n$. A \emph{rooted spanning forest} of $G$ is a set $\{T_j,j \in [k]\}$ of vertex-disjoint, rooted subtrees of $G$ with $\cup_{i = 1}^k V(T_i) = V(G)$. For a rooted tree $T$ we write $\rho(T)$ to denote the root of $T$. We always view the edges of a rooted tree as directed towards the root.

The \emph{Kingman coalescent on G} is defined as follows. Let $f_0$ be the empty rooted spanning forest of $G$, with $n$ elements, each of which is a rooted tree of size 1. For $i \geq 0$, if $\{\rho(T): T \in f_i\}$ is an independent set in $G$, then set $f_{i+1}=f_i$. 
Otherwise, there exists at least one edge connecting distinct roots of trees in $f_i$; choose one such edge $\{\rho(T),\rho(T')\}$ uniformly at random, orient it uniformly at random as $(\rho(T),\rho(T'))$, and add it to $f_i$ to form $f_{i+1}$. The unique tree in $f_{i+1}\setminus f_i$ has vertex set $v(T)\cup v(T')$ and root $\rho(T')$. Note that $f_m = f_{n-1}$ for all $m \geq n-1$. We write $F(G)=f_{n-1}$ for the final forest built by the process, which we call the \emph{Kingman forest} of $G$, and we write $(f_i,i \ge 0) \dist \king(G)$ for the process as a whole. 

Note that if $G=K_n$ is the complete graph with $n$ vertices, then to form $f_{i+1}$ from $f_i$, a uniformly random pair of trees of $f_i$ is chosen and merged. In this case, writing $\Pi_i$ for the partition of $V(G)$ formed by the vertex sets of the trees of $f_i$, then $(\Pi_i,0 \le i \le n-1)$ is distributed as the (discrete time) Kingman's coalescent on a set of size $n$. 
As such, the above process expands the traditional definition of Kingman's coalescent to a collection of processes in which some coalescent events may be forbidden.

\begin{figure}
    \centering

    \begin{tikzpicture}[scale=0.7]
    \begin{scope}[every node/.style={circle,draw}]
    \node[scale=0.8] (1) at (-4,0) {1};
    \node[scale=0.8] (2) at (-2,-1) {2};
    \node[scale=0.8] (3) at (-6,-1) {3};
    \node[scale=0.8] (4) at (-3,-2) {4};
    \node[scale=0.8] (5) at (-5,-2) {5};
    \node[scale=0.8] (1') at (4,0) {1};
    \node[scale=0.8] (2') at (6,-1) {2};
    \node[scale=0.8] (3') at (2,-1) {3};
    \node[scale=0.8] (4') at (5,-2) {4};
    \node[scale=0.8] (5') at (3,-2) {5};
    \end{scope}
    
    \node [left=1em] at (3){$G$,$F_0$:};

    \draw[ultra thick] (1) -- (2);
    \draw[ultra thick] (1) -- (4);
    \draw[ultra thick] (4) -- (5);
    \draw[ultra thick] (4) -- (3);
    \draw[ultra thick] (1) -- (3);
    
    \node [left=1em] at (3'){$G$,$F_1$:};
    \draw [->,ultra thick] (3') -- (1');
    
    \draw[ultra thick] (1') -- (2');
    \draw[ultra thick] (1') -- (4');
    \draw[ultra thick] (4') -- (5');
    \draw[ultra thick] (4') -- (3');
    \end{tikzpicture}

    \vspace{0.3cm}

    \begin{tikzpicture}[scale=0.7]
    \begin{scope}[every node/.style={circle,draw}]
    \node[scale=0.8] (1) at (-4,0) {1};
    \node[scale=0.8] (2) at (-2,-1) {2};
    \node[scale=0.8] (3) at (-6,-1) {3};
    \node[scale=0.8] (4) at (-3,-2) {4};
    \node[scale=0.8] (5) at (-5,-2) {5};
    \node[scale=0.8] (1') at (4,0) {1};
    \node[scale=0.8] (2') at (6,-1) {2};
    \node[scale=0.8] (3') at (2,-1) {3};
    \node[scale=0.8] (4') at (5,-2) {4};
    \node[scale=0.8] (5') at (3,-2) {5};
    \end{scope}
    
    \node [left=1em] at (3){$G$,$F_2$:};

    \draw[ultra thick] (1) -- (2);
    \draw[->,ultra thick] (4) -- (1);
    \draw[ultra thick] (4) -- (5);
    \draw[ultra thick] (4) -- (3);
    \draw[->,ultra thick] (3) -- (1);
    
    \node [left=1em] at (3'){$G$,$F_3$:};
    \draw [->,ultra thick] (3') -- (1');
    
    \draw[->,ultra thick] (1') -- (2') ;
    \draw[->,ultra thick] (4') -- (1');
    \draw[ultra thick] (4') -- (5');
    \draw[ultra thick] (4') -- (3');
    \draw[->,ultra thick] (3') -- (1');
    \end{tikzpicture}

    \vspace{0.3cm}

    \begin{tikzpicture}[scale=0.7]
    \begin{scope}[every node/.style={circle,draw}]
    \node[scale=0.8] (1') at (1,1.25) {1};
    \node[scale=0.8] (2') at (1,-0.5) {2};
    \node[scale=0.8] (3') at (-0.5,2.5) {3};
    \node[scale=0.8] (4') at (2.5,2.5) {4};
    \node[scale=0.8] (5') at (-1,0) {5};
    \end{scope}
    
    \node [left=2em,above=1.5em] at (5'){$F(G)$:};

    \draw[->,ultra thick] (1') -- (2');
    \draw[->,ultra thick] (4') -- (1');
    \draw[->,ultra thick] (3') -- (1');
    \end{tikzpicture}
    
    \caption{A realization of the Kingman coalescent on the graph $G$. The oriented edges are those in the forest $F_k$. After the third edge is added to $F_k$, there are no edges between the roots 2 and 5, so $F_k = F_3$ for all $k \geq 3$.}
    \label{fig:kingmanexample}
\end{figure}
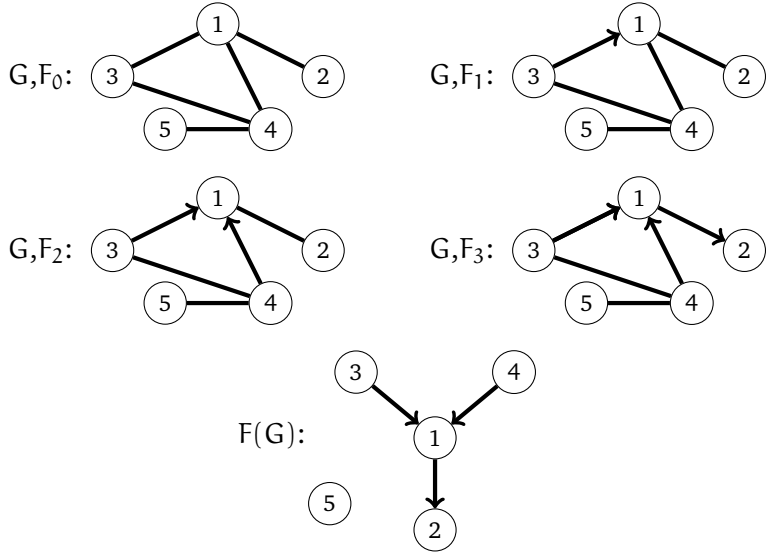

In this paper, we study the structure and number of trees of $F(G)$ when the underlying graph $G$ is an Erd\H{o}s-R\'enyi random graph $G_{n,p}$ for $p \in (0,1)$ fixed. Throughout the paper, we let $C_{n,p}$ denote the number of trees in $F(G_{n,p})$. Our main result is the following theorem.

\begin{theorem}\label{thm:trees}
    There exists a family of random variables $(C_p : p \in (0,1))$ such that
    \begin{enumerate}
        \item $C_{n,p} \convdist C_p$ and $\ex[C_{n,p}] \to \ex[C_p]$ as $n \to \infty$ for any fixed $p \in (0,1)$; and
        \item $\frac{p}{2(1-p)}C_p \convprob 1$ and $\frac{p}{2(1-p)}\ex[C_p] \to 1$ as $p \to 0$.
    \end{enumerate}
\end{theorem}

The main tools that we use in this proof also provide information in the case when $p \to 0$ as $n \to \infty$.

\begin{theorem}\label{thm:trees_small_p}
    Choose $p = p(n)$ such that $p \rightarrow 0$ and $np \rightarrow \infty$ as $n \rightarrow \infty$. Then $\frac{(1-p)}{2p}C_{n,p} \convprob 1$ and $\frac{(1-p)}{2p}\ex[C_{n,p}] \to 1$ as $n \to \infty.$
\end{theorem}

Our analysis of $\king(G_{n,p})$ uses a coupling with the Kingman coalescent on the complete graph, $\king(K_n)$, which is possible by the symmetry and independence of the existence of edges in $G_{n,p}$. This coupling, along with Theorem \ref{thm:trees}, allows us to deduce several structural properties of the resulting Kingman forests for the case when $p$ is fixed. We state the following theorem about the asymptotic behavior of tree sizes and heights as an illustrative example, but one could use the same connection to derive information about many other statistics.

In a rooted forest $F$, we define the height of a vertex $v$, $\height(v)$, to be its graph distance from the root of its tree. We set $\height(F) = \max_{v \in V(F)} \height(v)$, and call this quantity the height of the forest. For a fixed $k \geq 0$ and $\alpha_1,...,\alpha_k \in \N$, we say that a random vector $X = (X_1,...,X_k)$ has a Dirichlet distribution with parameters $\alpha_1,...,\alpha_k$, and write $X \dist \dirichlet(\alpha_1,...,\alpha_k)$ if it has a density function
    $$
    f_X(t_1,...,t_k) = \frac{\Gamma\left( \sum_{j=1}^k \alpha_j \right)}{\prod_{j=1}^k \Gamma(\alpha_j)}\prod_{j=1}^k x^{\alpha_j-1}
    $$
with respect to the Lebesgue measure on the unit simplex $(t_1,...,t_k)$, where $\Gamma$ denotes the standard gamma function.

\begin{theorem}\label{thm:manyproperties}
    For fixed $p \in (0,1)$, the following results hold.
    \begin{enumerate}
        \item Let $X_n = (|T_1|,...,|T_{C_{n,p}}|)$ be the sizes of the trees in $F(G_{n,p})$, listed in random order. Then, $(\frac{1}{n}X_n,C_{n,p}) \convdist (X,C_{p})$ as $n \to \infty$, where $C_{p}$ is the random variable from Theorem \ref{thm:trees} and, conditionally given $C_p$, $X$ has Dirichlet distribution with parameters $\alpha_1 = 1, \ldots, \alpha_{C_p} = 1$, i.e., $X$ has the uniform density $f_X(t_1,...,t_{C_p}) = \Gamma\left( C_p \right)\prod_{j=1}^{C_p} t_j^{0} = \Gamma\left( C_p \right)$ on the unit simplex.
        \item There exists $K > 0$ such that $|\ex[\height(F(G_{n,p}))] - e\log(n) + \frac{3}{2}\log\log(n)| \leq K$ for all $n$. Moreover, it holds that $\frac{\height(F(G_{n,p}))}{e\log(n)} \convprob 1$ as $n \to \infty$.
    \end{enumerate}
    
\end{theorem}

\subsection{Outline of the sections}

In Section \ref{sec:background} we motivate the work done in this paper and cover some background information on coalescing graph processes. In Section \ref{sec:edgereveal} we introduce and study the edge reveal process, a coupling between the Kingman coalescent and the underlying random graph that is key to our study of the Kingman coalescent. Section \ref{sec:trees} is dedicated to proving Theorem \ref{thm:trees} and \ref{thm:trees_small_p}. In Section \ref{sec:discussion} we describe the aforementioned coupling with $\king(K_n)$, and prove Theorem \ref{thm:manyproperties}. In Section \ref{sec:bound} we provide proof of some bounds which are used in the proof of Theorems \ref{thm:trees} and \ref{thm:trees_small_p}. Section \ref{sec:openquestions} concludes the paper with some open questions and ideas for future research.

\subsection{Notation}

Before moving forward we pause to collect some notation. For $x,y \in \R$, we define $x \vee y := \max\{x,y\}$ and $x \wedge y := \min\{x,y\}$. The set $\N$ denotes the natural numbers with $0$ excluded, and $\Z_{\geq0} = \N \cup \{0\}$. For a set $S$ and $k \in \Z_{\geq0}$, we let $\binom{S}{k}$ denote the collection of all subsets of $S$ of size exactly $k$. For $k \in \N$, we write $[k]:=\{1,...,k\}$. For an undirected graph $G = (V,E)$ and a vertex $v \in V$, we define $\deg_G(v) = |\{e \in E : u \in e\}|$. We write $E(G)$ to refer to the edge set of a graph $G$ and $V(G)$ to refer to its vertex set. If we say ``$G$ is a graph on $V$'', we mean that $G$ is a graph with $V(G)=V$. For a graph $G=(V,E)$ and $e \in {V\choose 2}$, we write $G+e$ for the graph $(V,E\cup\{e\})$; if $e \in E$ then $G+e=G$. A rooted tree $t=(V(t),E(t))$ is {\em increasing} if $V(t)\subset \N$ and vertex labels increase along any root-to-leaf path (equivalently, if every non-root vertex's label is strictly larger than that of its parent. Finally, $\rho(F)$ is the set of all the roots of trees in the rooted forest $F$.

We use $\dist$ to denote distributional equality, $\convdist$ to denote convergence in distribution, and $\convprob$ to denote convergence in probability. For two random variables $X$ and $Y$, we say that $X$ stochastically dominates $Y$, writing $Y \preceq X$, if $\prob(X \geq x) \geq \prob(Y \geq x)$ for all $x \in \R$. For a finite set $S$, we say that $X \dist \unif(S)$ if for all $s \in S$, $\prob(X = s) = |S|^{-1}$. We say that $X$ is a geometric random variable with parameter $p$, and write $X \dist \geo(p)$, if $\prob(X = k) = (1-p)^kp$ for $k \in \Z_{\geq0}$. We say that $X$ has a negative binomial distribution with parameters $r \in \N$ and $p \in (0,1]$, and write $X \dist \negbin(r,p)$, if $\prob(X = k) = \binom{k+r-1}{k}(1-p)^kp^r$ for $k \in \Z_{\geq 0}$. We say that $X$ has a hypergeometric distribution with parameters $n \in \Z_{\geq 0}, m \in [n], k \in[n]$, and write $X \dist \hyper(k,m,n)$, if

    $$
    \prob(X = j) = \frac{\binom{m}{j}\binom{n-m}{k-j}}{\binom{n}{k}}.
    $$

Throughout the article, we let $\mathcal{F}_{n,k}$ be the set of rooted forests on $[n]$ with $k$ edges that are each given a unique label in $[k]$, such that edge labels decrease along all root-to-leaf paths. We let $\mathcal{G}_{n,k}$ denote the set of graphs on $[n]$ with $k$ edges. For $S \subseteq [n]$, we define $\cG_{S,k}$ to be the set of all graphs on $S$ with $k$ edges. 
We let $\cG_{n,k}^{(r)} = \cup_{S\subseteq[n] :|S|=r} \cG_{S,k}$ be all graphs with $k$ edges and a vertex set of size~$r$ drawn from the set $[n]$.

\section{Background and motivations}\label{sec:background}

An \emph{$n$-coalescent} is a stochastic process $(P_k)_{k=0}^\infty$ consisting of partitions of $[n] = \{1 , ... , n\}$, where $P_0 = \{\{1\} , ... , \{n\}\}$ and $P_{k+1}$ is derived from $P_k$ by merging two distinct portions $A$ and $B$ with probability proportional to some function $\kappa(|A|,|B|)$. Three particularly well studied examples are $\kappa(x,y) = 1$, $\kappa(x,y) = x + y$, and $\kappa(x,y) = xy$. These choices are referred to as Kingman's coalescent, the additive coalescent, and the multiplicative coalescent respectively \cite{kingman1982coalescent,pitman1999coalescent,aldous1997brownian}.

The study of $n$-coalescent processes has motivations coming from across the sciences \cite{berestycki2009recent}. Some of the early mathematical work on coalescent models was due to Kingman \cite{kingman1982genealogy,kingman1982coalescent}, with motivation coming from the area of population genetics. Since then, coalescent processes have become part of the standard toolkit of population genetics for studying ancestral recombination graphs \cite{cousinsStructuredCoalescentModel2025, nielsenInferenceApplicationsAncestral2025}. For a second source of inspiration one can look towards statistical physics, where coalescent processes have naturally emerged within the study of spin glasses \cite{bolthausen1998ruelle,pitman1999coalescents,goldschmidt2005random}.

The three coalescents mentioned above are often viewed as a sequence of forests $(F_k)_{k=0}^\infty$ on the vertex set $[n]$ \cite{addario2015partition} with $F_0$ being $([n],\emptyset)$. For the Kingman coalescent, the sequence exactly corresponds to $\king(K_n)$. For the additive coalescent, we sample a uniform pair $(x,y)$ such that $y \in [n]$ and $x$ is the root of a tree in $F_k$ that does not contain $y$. Then, $F_{k+1}$ is formed by adding the edge $(x,y)$ to $F_k$, which results in $x$ no longer being a root. The multiplicative coalescent is typically seen as a sequence of unrooted forests where $F_{k+1}$ is derived from $F_k$ by adding a uniform edge to $F_k$ from among edges whose addition would not create a cycle.

Coalescent graph processes have frequently appeared in the random graph theory literature. Various versions of the Kingman coalescent \cite{kingman1982coalescent} have been used to study recursively growing random trees via direct distributional equivalences \cite{doyle1976linear,addario2018high,bellin2023uniform}. The additive coalescent has appeared naturally in the study of uniform trees, as the forest valued version of the process produces a tree that is distributed uniformly over all labelled trees \cite{pitman1999coalescent,addario2015partition}. The multiplicative coalescent has appeared in both the study of component sizes in critical Erd\"{o}s-R\'{e}nyi random graphs as well as the study of minimum spanning trees \cite{aldous1997brownian,addario2009critical,addario2017scaling}.

These three coalescents have all been studied in depth when there are no ``external'' constraints, in the sense that all mergers permitted by the coalescent rule in question are permitted. However, only a small amount of work has been put towards understanding the forests that emerge when we add the restriction that all edges must come from a set of allowed edges $E$, i.e., when we run the coalescents on an underlying graph $G$. For all three coalescents, the size and structure of the forests may be greatly affected by structure of the underlying graph, and this is a primary motivation for our investigation of the structure of $\king(G_{n,p})$.  There is some work on the structure of the multiplicative coalescent in non-complete geometries, due its connection with minimum spanning trees  \cite{addario2017scaling,addario2021geometry,MR3857861}. Thus far, we are unaware of any research into the structure of additive coalescents on non-complete graphs.

\section{The edge reveal process}\label{sec:edgereveal}

When $G \dist G_{n,p}$, there is a useful Markov chain which couples the construction of $\king(G)$ to a construction of $G$ itself. We call this coupling the \emph{edge reveal process}; we shall use it to analyse the number of trees in $\king(G)$. 

\subsection{Definition and distributional identities}

Let $B = \left\{B_{e} : e \in \binom{[n]}{2}\right\}$ be a collection of independent $\ber(p)$ random variables, so if $E = \left\{e \in \binom{[n]}{2}: B_e = 1\right\}$, then the graph $([n],E)$ is distributed like $G_{n,p}$. Independent of $B$, let $(e_k)_{k=1}^\infty$ be a sequence of independent $\unif \binom{[n]}{2}$ random variables. We call $B$ the \emph{bits} of the edge reveal process and the sequence $(e_k)_{k=0}^\infty$ the \emph{queried pairs} of the edge reveal process. We set $R_0 = [n]$, $F_0 = ([n] , \emptyset)$, and $G_0 = ([n],\emptyset)$. Then, for $k \geq 0$ we inductively define $(R_{k+1},F_{k+1},G_{k+1})$, and $L_{k+1}:E(F_{k+1}) \to \N$, as:
\begin{enumerate}
	\setlength{\itemsep}{0pt}
	\item If $B_{e_{k+1}} = 0$: set $(R_{k+1} , F_{k+1} , G_{k+1}) = (R_k , F_k , G_k)$.
	\item If $B_{e_{k+1}} = 1$ and $e_{k+1} \not\subseteq R_k$: Set $(R_{k+1} , F_{k+1} , G_{k+1}) = (R_k , F_k , G_k + e_{k+1})$. (Note that $e_{k+1}$ may already be in $G_k$; the graph does not change in this case).
	\item If $B_{e_{k+1}} = 1$ and $e_{k+1} \subseteq R_k$: let $O_{k+1} = (u_{k+1} , v_{k+1})$ be a uniformly random orientation of $e_{k+1}$. Define $(R_{k+1} , F_{k+1} , G_{k+1}) = (R_k \setminus u_{k+1} , F_{k} + O_{k+1} , G_k + e_{k+1})$ and let $L_{k+1}$ be such that $L_{k+1}(e_{j}) = |E(F_j)| + 1$ for all $1 \leq j \leq k+1$. (Note that with this labelling convention, the labelled forests $(F_k, L_k)$ are elements of $\cF_{n,|E(F_k)|}$).
\end{enumerate}
We write $(R_k,F_k,G_k)_{k=0}^\infty \dist \erp(n,p)$. This sequence is infinite, though once all pairs in $\binom{[n]}{2}$ have been queried, the sequence never changes. Let $\tau_0 = 0$, and let $\tau_k = \inf\{j > \tau_{k-1} : F_j \ne F_{j-1}\}$ for all $k \geq 0$ be the times when updates of type (iii) happen (note that these times can be infinite). The snapshots of the sequence $(R_k,F_k,G_k)$ at the times $\tau_0,...,\tau_{n-1}$ are the main points of interest. We call $\tau_0, ... , \tau_{n-1}$ the \emph{coalescing times} of $\erp(n,p)$.

\begin{lemma}\label{lem:erp=kingman}
    Let $(R_k,F_k,G_k)_{k=0}^\infty \dist \erp(n,p)$ and set $F_k^* = F_{\tau_k \wedge \tau_{k^*}}$ for all $0 \leq k \leq \infty$. Then, $(F_k^*)_{k=0}^{\infty} \dist \king(G_{n,p})$. 
\end{lemma}

\begin{proof}
The forest $F_{\tau_{j+1} \wedge \tau_{k^*}}^*$ is obtained from $F_{\tau_{j} \wedge \tau_{k^*}}^*$ by adding a single new edge that is sampled uniformly from the set 
$$
\left\{e \in \binom{[n]}{2} : e \subseteq \rho(F_{\tau_{j-1}}) , \ B_e = 1\right\},
$$
with uniformly random orientation.
This is exactly the rule for how the edges are added for the Kingman coalescent. Since each bit $B_e$ is $\ber(p)$ distributed, we have $\big([n],\{ e \in \binom{[n]}{2} : B_e = 1 \}\big) \dist G_{n,p}$. Thus, $(F_k^*)_{k=0}^{\infty} \dist \king(G_{n,p})$.
\end{proof}

For the next lemma we introduce the \emph{complement process} of $\erp(n,p)$, which is the sequence $(G_k^*)_{k=0}^\infty$ given by $G^*_k = (\rho(F_k) , \{e_1 , ... , e_k\}\cap \binom{\rho(F_k)}{2} \setminus E(G_k))$. The vertex set of $G_k^*$ is the set of roots of $F_k$; its edges are exactly the pairs of roots that have been queried by time $k$ and are {\em not} in $G_k$. That is, all the edges in $G_k^*$ that have been ``verified'' to not be in the underlying graph by time $k$. We often refer to edges that have $B_e = 0$ (and thus, all edges in $G_k^*$ for any $k \geq 0$) as the \emph{non-edges} of the edge reveal process, and we see updates of type (i) as revealing a non-edge.

In Section \ref{sec:trees}, we use the sequence $(G_k^*)_{k=0}^\infty$ to study the number of trees in a Kingman forest of $G_{n,p}$. A very useful property for this purpose is that $F_k$ and $G_k^*$ are each uniformly random conditional on their number of edges. In particular, this implies that their structure is uniform at coalescing times.

\begin{lemma}\label{lem:uniform}
Let $(R_k,F_k,G_k)_{k=0}^\infty \dist \erp(n,p)$. Let
    $$
    \cS(m , \ell) = \left\{ (f,g) : f \in \cF_{n,m}, \ g \in \cG_{\rho(f),\ell} \right\}.
    $$
Then for any $0 \leq m + \ell \leq k$ and any $(f,g) \in \cS(m,\ell)$, we have 
$$
\prob\Big(F_k = f, \ G_k^* = g \ \big| \ |E(F_k)| = m, |E(G_k^*)| = \ell \Big) = \frac{1}{|\cS(m,\ell)|}.
$$
Thus,
    \begin{enumerate}
        \setlength{\itemsep}{0pt}
        \item for any $f \in \cF_{n,m}$ with $m \leq k$, $\prob( F_k = f \ | \ |E(F_k)| = m) = 1/|\cF_{n,m}|$; and 
        \item for any $g \in \cG_{n,m}^{(r)}$ with $m \leq k$, $\prob( G_k^* = g \ | \ |E(G_k^*)| = m , \ |R_k| = r) = 1/|\cG_{n,m}^{(r)}|$.
    \end{enumerate}
\end{lemma}

\begin{proof}

We prove the first identity via induction on $k$. Instead of showing it directly, we argue that $\prob(F_k = f , \ G_k^* = g) = \prob(F_k = \hat{f} , \ G_k^* = \hat{g})$ for any $(f,g),(\hat{f},\hat{g}) \in \cS(m,\ell)$ with $m,\ell \geq 0$ arbitrary. The base case is immediate as $F_0$ and $G_0^*$ are deterministic. Suppose that the identity holds for some $k \geq 0$, and let $f,g \in \cS(m,\ell)$ for some arbitrary $m,\ell \geq 0$. Let $(u,v)$ be the edge in $f$ of largest label. Consider the event $\{F_{k+1} = f, G_{k+1}^* = g\}$. Exactly one of the three following (disjoint) events must occur if $\{F_{k+1} = f, G_{k+1}^* = g\}$ is to occur:
    \begin{enumerate}
        \setlength{\itemsep}{0pt}
        \item $(F_k,G_k^*) = (f,g)$: In this case we have that $(F_{k+1},G_{k+1}^*) = (f,g)$ if and only if $e_{k+1} \in E(f) \cup E(g) \cup \binom{[n] \setminus \rho(f)}{2}$.
        \item $(F_k,G_k^*) = (f \setminus (u,v) , g')$, where $g'$ is a graph on the vertex set $\rho(f) \cup \{u\}$ and edge set $E(g) \cup S$ for $S \subseteq \{\{u,w\}:w \in \rho(f)\setminus \{u,v\}\}$: In this case we have that $(F_{k+1},G_{k+1}^*) = (f,g)$ if and only if $e_{k+1} = \{u,v\}$, $B_{\{u,v\}} =1$, and the orientation chosen for the edge is $O_{k+1} = (u,v)$.
        \item $(F_k,G_k^*) = (f,g \setminus e)$ for some $e \in E(g)$: In this case we have that $(F_{k+1},G_{k+1}^*) = (f,g)$ if and only if $e_{k+1} = e$ and $B_{e} = 0$.
    \end{enumerate}
By the definition of the edge reveal process, one can see that there are no other possibilities for $(F_k,G_k^*)$ which allow for $(F_{k+1},G_{k+1}^*) = (f,g)$ to occur. Denote the set of graphs $g'$ in (ii) by $S(g)$, and note that $|V(g')|=|V(g)\cup \{u\}|=m+1$ for all $g' \in S(g)$. Let
    \begin{align*}
    A(f,g) &= \{ (F_k,G_k^*) = (f,g) \} \bigcap \left\{e_{k+1} \in E(f) \cup E(g) \cup \binom{[n] \setminus \rho(f)}{2}\right\}, \\
    B(f,g) &= \bigcup_{g' \in S(g)}\big\{ (F_k,G_k^*) = (f \setminus (u,v) , g') \big\} \cap \big\{ e_{k+1} = \{u,v\} , \ B_{\{u,v\}} =1 , \ O_{k+1} = (u,v) \big\}, \\
    C(f,g) &= \bigcup_{e \in E(g)}\big\{ (F_k,G_k^*) = (f , g \setminus e) \big\} \cap \{ e_{k+1} = e , \ B_e = 0 \}
    \end{align*}
be the events from (i), (ii), and (iii). Then,
    $$
    \prob(F_k = f , \ G_k^* = g) = \prob(A(f,g)) + \prob(B(f,g)) + \prob(C(f,g)).
    $$
Via explicit counting, we  obtain the following identities:
    \begin{align*}
    \prob(A(f,g)) &= \prob(F_k = f, \ G_k^* = g)\prob\left( e_{k+1} \in E(f) \cup E(g) \cup \binom{[n] \setminus \rho(f)}{2} \right) \\
    &= \prob(F_k = f, \ G_k^* = g)\frac{|E(f)| + |E(g)| + \binom{n - |\rho(f)|}{2}}{\binom{n}{2}}, \\
    \prob(B(f,g)) &= \sum_{g' \in S(g)} \prob(F_k = f, \ G_k^* = g')\prob(e_{k+1} = \{u,v\} , \ B_{\{u,v\}} = 1, \ O_{k+1} = (u,v)) \\
    &= \sum_{g' \in S(g)}\prob(F_k = f, \ G_k^* = g') \cdot \frac{p}{2\binom{n}{2}}, \\
    \prob(C(f,g)) &= \sum_{e \in E(g)}\prob(F_k = f, \ G_k^* = g \setminus e)\prob(e_{k+1} = e, \ B_e = 0) \\
    &= |E(g)| \cdot \prob(F_k = f, \ G_k^* = g \setminus e) \cdot \frac{1-p}{\binom{n}{2}}.
    \end{align*}
From here, induction yields that $\prob(A(f,g)) = \prob(A(\hat{f},\hat{g}))$ and $\prob(C(f,g)) = \prob(C(\hat{f},\hat{g}))$ for any other pair $(\hat{f},\hat{g}) \in S(m,\ell)$. For $B(f,g)$, to see that $\prob(B(f,g)) = \prob(B(\hat{f},\hat{g}))$, we remark that there are exactly ${n-m \choose j}$ graphs $g' \in S(g)$ such that $|E(g')| = |E(g)|+j$ for all $0 \le j \le k-1$. Since this does not depend on the structure of $g$, only on its numbers of vertices and edges, we can apply the inductive hypothesis again to get that $\prob(B(f,g)) = \prob(B(\hat{f},\hat{g}))$. 

The first identity in the lemma follows from the fact that $\prob(F_k = f , \ G_k^* = g) = \prob(F_k = f' , \ G_k^* = g')$ for all $(f,g), \ (f',g') \in S(m,\ell)$. The second and third identities follow straightforwardly from the first and the fact that $\prob((F_k,G_k^*) \in S(m,\ell)) > 0$ if $m + \ell \leq k$.
\end{proof}

\subsection{Edge counts and monotonicity in the complement process}

Tracking the number of edges in the complement process of $\erp(n,p)$, which we do in Section \ref{sec:trees}, is a key part of the analysis of the number of trees in $F(G_{n,p})$. Write $N_k=|E(G_k^*)|$. Since $N_k$ is the number of pairs in ${R_k \choose 2}$ that have been verified to be non-edges by step $k$, if $N_k={R_k \choose 2}$ then the edge reveal process has terminated in the sense that $F_k = F_{j}$ for all $j \geq k$. Let $K^* = \inf\{ k \geq 0 : N_k \geq \binom{R_k}{2} \}$. We define $(M_k)_{k=0}^{n-1} := (N_{\tau_k \wedge K^*})_{k=0}^{n-1}$, which we call the \emph{edge count walk} of the edge reveal process. By Lemma \ref{lem:erp=kingman} and the discussion above, if $J^* := \inf\{ j \geq 0 : M_{j} \geq \binom{n-j}{2}\}$, then $n-J^*+1 \dist C_{n,p}$. The plus one appears because in the step where we terminate, a vertex does not get removed.

We now turn our attention to describing a single step, $M_{k+1}-M_k$, of the edge-count walk. First suppose that $\tau_k < \infty$. For all $0 \leq k \leq n-2$, let $S_k = \{\tau_k+1,...,K^* \wedge (\tau_{k+1})\}$ and let
    $$
    X_k = \left|\left\{ j \in S_k : e_j \in \binom{R_{\tau_k}}{2} , \ B_{e_j} = 0, \ e_j \notin \{ e_1 , ... , e_{j-1} \} \right\}\right|,
    $$
The random variable $X_k$ is precisely the number 
of pairs that are verified to be non-edges between times $\tau_k$ and $K^*\wedge \tau_{k+1}$, 
so if $|E(G_{\tau_k}^*)| < {n-k \choose 2}$ then $K^* \wedge \tau_{k+1}>\tau_k$ and  $X_k=|E(G_{K^*\wedge \tau_{k+1}^*-1})|-|E(G_{\tau_k}^*)|$. On the other hand, if $|E(G_{\tau_k}^*)| = {n-k \choose 2}$ then either $\tau_k=\infty$ or else $K^*=\tau_k$, and in either case $X_k=0$. Moreover, provided that $|E(G_{\tau_k}^*)| < {n-k \choose 2}$, by the definition of the process, the number of pairs that are verified to be non-edges between times $\tau_k$ and $K^*\wedge \tau_{k+1}$ is distributed as a $\geo(p)$ random variable truncated at ${n-k \choose 2}-|E(G_{\tau_k}^*)|$. 

Let $Y_k = \deg_{G_{(\tau_{k+1} - 1)}^*}(u_{k+1}) \1_{\{ \tau_{k+1} < \infty\}}$, where $u_{k+1}$ is the tail of the oriented edge $O_{k+1}$. When $\tau_{k+1} = \infty$, $Y_k$ is defined to be zero, and so the fact that the orientation $O_{k+1}$ doesn't exist is not a problem for the definition. When $\tau_k = \infty$ we set $X_k = Y_k = 0$. By the definition of the edge reveal process we have that $M_{k+1}= M_k + X_k - Y_k$ for all $0 \leq k \leq n-2$. 

\begin{lemma}\label{lem:recursion}
    For all $0 \leq k \leq n-2$ we have the following:
    \begin{enumerate}
        \item Conditionally given $M_k$, $X_k$ is a $\geo(p)$ random variable truncated at $\binom{n-k}{2}-M_k$.
        \item Conditionally given $M_k$ and $X_k$, $Y_k$ is distributed like
        \begin{equation}\label{eq:Mkjumps}
        \hyper\left(n-k-2 , M_k + X_k , \binom{n-k}{2} - 1\right)\1_{\{ M_k+X_k \leq \binom{n-k}{2}-1 \}}.
        \end{equation}
    \end{enumerate}
\end{lemma}

\begin{proof}
The first identity was verified prior to the proof, so we only need to prove the second one. For this, note that that $\tau_{k+1} = \infty$ if and only if $X_k \geq \binom{n-k}{2} - M_k$ and that no vertex is removed from the complement process after time $\tau_k$ if $\tau_{k+1} = \infty$. Since $Y_k = \deg_{G_{(\tau_{k+1} - 1)}^*}(u_{k+1}) \1_{\{ \tau_{k+1} < \infty\}}$, this fact explains the indicator in (\ref{eq:Mkjumps}). When $X_k < \binom{n-k}{2} - M_k$, it holds that $\tau_{k+1} < \infty$. Letting $e_{\tau_{k+1}} =\{u,v\}$, by Lemma \ref{lem:uniform}, conditionally given $M_k$ and $X_k$, $G_{(\tau_{k+1} - 1)}^*$ is distributed uniformly over the set
    $$
    \left\{g \in \cG_{n,M_k+X_k}^{(n-k)} : \{u,v\} \notin g\right\}.
    $$
From here, recalling the well known fact that the degree of a typical vertex in a graph drawn uniformly from $\cG_{n,k}^{(r)}$ has a $\hyper(k,n-1,\binom{n}{2})$ distribution justifies the first factor in (\ref{eq:Mkjumps}), as we are conditioning on the edge $\{u_{k+1},v_{k+1}\}$ to not be in $G_{(\tau_{k+1} - 1)}^*$.
\end{proof}

Using the relationship between $M_k,X_k,$ and $Y_k$ that is derived in Lemma \ref{lem:recursion} we can show that the edge count walk is monotone with respect to $n$.

\begin{lemma}\label{lem:monotonicityofedgecount}
    Let $m \leq n$, and let $(M_k^{(n)})_{k=0}^{n-1}$ and $(M_k^{(m)})_{k=0}^{m-1}$ be the edge count walks for two edge reveal processes, $\erp(n,p)$ and $\erp(m,p)$ respectively. Then, $M_k^{(m)} \preceq M_{k + (n-m)}^{(n)}$ for all $0 \leq k \leq m-1$.
\end{lemma}

In the proof of the above lemma, we use the following property about hypergeometric random variables.

\begin{lemma}\label{lem:hypergeo_dom}
    $X \dist \hyper(k,m,n)$ and $Y \dist \hyper(k,m',n)$, where $0 \leq m \leq m' \leq n$. Then, $m - X \preceq m' - Y$.
\end{lemma}

\begin{proof}
Consider a population of $n$ balls, with $m$ coloured red, $m'-m$ coloured blue, and the rest coloured black. Draw $k$ balls from the population without replacement and let $X$ be the number of red balls drawn and $Y$ the number of red or blue balls drawn. Then, $Y - X \leq m'-m$, and so $m - X \leq m' - Y$.
\end{proof}

\begin{proof}[Proof of Lemma \ref{lem:monotonicityofedgecount}]
    Let the coalescing times of the respective processes be $(\tau_k^{(n)})_{k=0}^{n-1}$ and $(\tau_k^{(m)})_{k=0}^{m-1}$. The argument proceeds via induction, with the base case being immediate from the fact that $M_{0}^{(m)} = 0$ deterministically. Suppose that $M_{k}^{(m)} \preceq M_{k + (n-m)}^{(n)}$ for some $0 \leq k \leq m-2$. Let $(X_k^{(n)},Y_k^{(n)})_{k=0}^{n-1}$ and $(X_k^{(m)},Y_k^{(m)})_{k=0}^{m-1}$ be defined as before Lemma \ref{lem:recursion} for $\erp(n,p)$ and $\erp(m,p)$ respectively. 
    
    First suppose that $\tau_{k+(n-m)}^{(n)} = \infty$. In this case, we have $M_{k+1+(n-m)}^{(n)} = M_{k+(n-m)}^{(n)} \geq \binom{m-k}{2}$. If $\tau_k^{(m)} = \infty$, then we can, by induction, take a coupling such that $M_{k+1}^{(m)} = M_{k}^{(m)} \leq M_{k+(n-m)}^{(n)} = M_{k+1+(n-m)}^{(n)}$. If $\tau_k^{(m)} < \infty$, then we necessarily have $M_{k+1}^{(m)} \leq \binom{m-k}{2} \leq M_{k+1+(n-m)}^{(n)}$.
    
    Now assume $\tau_{k+(n-m)}^{(n)} < \infty$ (and hence, $\tau_k^{(m)} < \infty$ as well). Recall from Lemma \ref{lem:recursion} that, under this conditioning, $X_k^{(m)}$ is distributed like $G \wedge (\binom{m-k}{2} - M_{k}^{(m)})$ and $X_{k+(n-m)}^{(n)}$ like $G \wedge (\binom{m-k}{2} - M_{k+(n-m)}^{(n)})$ for $G \dist \geo(p)$ sampled independently of $M_k^{(m)}$ and $M_{k+(n-m)}^{(n)}$. Take some coupling where $M_{k}^{(m)} \leq M_{k + (n-m)}^{(n)}$ and let $G \dist \geo(p)$ be independent. Then we have
    \begin{align*}
        &M_{k + (n-m)}^{(n)} + X_{k+(n-m)}^{(n)} \\
        =& M_{k}^{(m)} + (M_{k + (n-m)}^{(n)} - M_{k}^{(m)}) + G \wedge \left(\binom{m-k}{2} - M_{k}^{(m)} - (M_{k + (n-m)}^{(n)} - M_{k}^{(m)})\right) \\
        \geq& M_{k}^{(m)} + (M_{k + (n-m)}^{(n)} - M_{k}^{(m)}) + G \wedge \left(\binom{m-k}{2} - M_{k}^{(m)}\right) - (M_{k + (n-m)}^{(n)} - M_{k}^{(m)}) \\
        =& M_{k}^{(m)} + X_k^{(m)},
    \end{align*}
    proving that $M_{k}^{(m)} + X_k^{(m)} \preceq M_{k + (n-m)}^{(n)} + X_{k+(n-m)}^{(n)}$.
    
    If $\tau_{k+1}^{(m)} = \infty$, then $\tau_{k+1+(n-m)}^{(n)} = \infty$ as well and so $Y_k^{(m)} = Y_{k + (n-m)}^{(n)} = 0$, which would complete the proof. If $\tau_{k+1+(n-m)}^{(n)} = \infty$ and $\tau_{k+1}^{(m)} < \infty$, then we can similarly finish the proof immediately as $Y_k^{(m)} \geq Y_{k + (n-m)}^{(n)} = 0$ in this case. Hence, we can suppose that $\tau_{k+1}^{(m)},\tau_{k+1+(n-m)}^{(n)} < \infty$. In this case, conditionally given $M_k^{(m)}$, $M_{k+(n-m)}^{(n)}$, $X_k^{(m)}$, and $X_{k+(n-m)}^{(n)}$, we have that $Y_{k+(n-m)}^{(n)} \dist \hyper(m-k-2 , M_{k+(n-m)}^{(n)} + X_{k+(n-m)}^{(n)} , \binom{m-k}{2})$ and $Y_k^{(n)} \dist \hyper(m-k-2 , M_{k}^{(m)} + X_k^{(m)} , \binom{m-k}{2})$. Since $M^{(n)}_{k+1+(n-m)} = (M_{k+(n-m)}^{(n)} + X_{k+(n-m)}^{(n)}) - Y_{k+(n-m)}^{(n)}$ and $M^{(m)}_{k+1} = (M_{k}^{(m)} + X_{k}^{(m)}) - Y_{k}^{(m)}$, the result then follows from Lemma \ref{lem:hypergeo_dom}.
\end{proof}

Recall the fact, stated in discussion at the start of this subsection, that $C_{n,p} \dist n-J^* + 1$, where $J^* := \inf\{ j \geq 0 : M_{j} \geq \binom{n-j}{2}\}$. Combining these definitions with the above lemma directly implies an important monotonicity result for the Kingman coalescent on $G_{n,p}$.

\begin{corollary}\label{cor:monotonicityofkingman}
    Let $n \geq m \geq 0$. Then, $C_{m,p} \preceq C_{n,p}$.
\end{corollary}

\section{The number of trees in $F(G_{n,p})$}\label{sec:trees}

In this section, we prove Theorems \ref{thm:trees} and \ref{thm:trees_small_p}. By definition, obtaining results on the number of trees is essentially equivalent to obtaining results on $J^*$, which requires some bounds on $(M_k)_{k=0}^{n-1}$. We postpone the proofs until Section \ref{sec:bound}, though we record the results now. For the rest of the paper, we introduce the notation $K_{p,\epsilon}^{-} = \lfloor\frac{2(1-\epsilon)(1-p)}{p}\rfloor$ and $K_{p,\epsilon}^+ = \lceil \frac{2(1+\epsilon)(1-p)}{p}\rceil$.
\begin{lemma}\label{lem:Mk_results}
 For any $\eta,\epsilon,\delta \in (0,1)$, there exist $C,L,c > 0$ such that the following hold:
\begin{enumerate}
    \setlength{\itemsep}{1em}
    \item  Fix $p \in (0,\eta)$ and integers $n$ and $\ell$ such that $n \geq \ell \geq L \vee K_{p,\epsilon}^+$. Then,
        $$
        \prob\left( \bigcup_{k = 1}^{n-\ell} \left\{ M_k \geq (1+\epsilon)\frac{(1-p)(n-k)}{p} \right\} \right) \leq Ce^{-c\ell}.
        $$
    \item For any $n \geq 0$ and any $p \in (0,\eta)$ such that $K_{p,\epsilon}^- \geq L$, we have,
    $$
        \prob\left(M_{n-K_{p,\epsilon}^-} \leq \binom{K_{p,\epsilon}^-}{2}\right)  \leq \delta + \frac{1}{2\epsilon}\left( \frac{K_{p,\epsilon}^+-1}{n-1} \right).
    $$
\end{enumerate}
\end{lemma}

A brief computation shows that $\frac{(1+\epsilon)(1-p)(n-k)}{p}$ crosses above $\binom{n-k}{2}$ around the value $k$ for which $n-k = K_{p,\epsilon}^+$. Thus, Lemma \ref{lem:Mk_results} tells us that $J^*$ is likely to be between 
$K_{p,\epsilon}^-$ and $K_{p,\epsilon}^+$ for $n$ large. Using Lemma \ref{lem:Mk_results} we can make this intuition into a quantitative result.

\begin{lemma}\label{lem:Cnp_bounds}
    Let $p = p(n) < 1$ be such that $np \to \infty$ as $n \to \infty$ and $\limsup_{n \to \infty} p(n) < 1$. For any $\delta, \epsilon \in (0,1)$, there exists $L \geq 0$ such that, if $\liminf_{n \to \infty} K_{p,\epsilon}^- \geq L$, then
    $$
    \limsup_{n \to \infty}\left| \frac{p\ex[C_{n,p}]}{2(1-p)} - 1 \right| \leq \delta,
    $$
    and
    $$
    \limsup_{n \to \infty}\prob\left( \left|C_{n,p} - \frac{2(1-p)}{p}\right| \geq \frac{2\epsilon(1-p)}{p} \right) \leq \delta.
    $$
\end{lemma}

\begin{proof}
    Since $C_{n,p} \dist n-J^*+1$, we have
\begin{equation}\label{eq:expec_C_and_J_star}
\ex[C_{n,p}-1] = \sum_{k=0}^{n-1} \prob(C_{n,p} \geq k+1) = \sum_{k=0}^{n-1} \prob\left( J^* \leq n-k \right).
\end{equation}
Using the definition of $J^*$, we obtain
\begin{equation}\label{eq:lower_bound_for_Cnp_expectation}
    \ex[C_{n,p}-1] \geq \sum_{k=1}^{K_{p,\epsilon}^-} \prob\left( \bigcup_{j=0}^{n-k}\left\{ M_{j} \geq \binom{n-j}{2}\right\} \right) \geq K_{p,\epsilon}^-\prob\left(M_{n-K_{p,\epsilon}^-} \geq \binom{K_{p,\epsilon}^-}{2} \right).
\end{equation}
By Lemma \ref{lem:Mk_results} (ii), there exists $L_1 \geq 0$ such that for any $n$ and $p$ satisfying $K_{p,\epsilon}^- \geq L_1$ we have
\begin{equation}\label{eq:lower_bound_for_Mk_in_Cnp_proof}
\prob\left( M_{n-K_{p,\epsilon}^-} \geq \binom{K_{p,\epsilon}^-}{2} \right)\geq 1-\delta - \frac{1}{2\epsilon}\left(\frac{K_{p,\epsilon}^+-1}{n-1}\right).
\end{equation}
Note that, by our assumptions on $n$ and $p$, it holds that $\lim_{n \to \infty} K_{p,\epsilon}^-n^{-1} = 0$. Using this fact along with (\ref{eq:lower_bound_for_Cnp_expectation}) gives
\begin{equation*}
\liminf_{n \to \infty} \frac{\ex[C_{n,p}-1]}{K_{p,\epsilon}^-} \geq \liminf_{n \to \infty}\left(1-\delta - \frac{1}{2\epsilon}\left(\frac{K_{p,\epsilon}^+-1}{n-1}\right)\right) = (1-\delta),
\end{equation*}
whenever $\liminf_{n \to \infty} K_{p,\epsilon}^- \geq L_1$. By increasing $L_1$ if needed, using the definition of $K_{p,\epsilon}^-$, we obtain the bound,
\begin{equation}\label{eq:liminf_Cnp}
    \liminf_{n \to \infty} \frac{2(1-p)\ex[C_{n,p}]}{p} \geq (1 - \delta)^2(1-\epsilon).
\end{equation}
On the other hand, from (\ref{eq:expec_C_and_J_star}) and the definition of $C_{n,p}$, we also have
\begin{equation*}
    \ex[C_{n,p}-1] \leq K_{p,\epsilon}^+ + \sum_{k = K_{p,\epsilon}^++1}^{n-1} \prob(C_{n,p} \geq k+1) = K_{p,\epsilon}^+ +  \sum_{k = K_{p,\epsilon}^++1}^{n-1} \prob\left( M_{n-k} \geq \binom{k}{2} \right).
\end{equation*}
Since, for $k \geq K_{p,\epsilon}^+ +1$ we have
\begin{equation}\label{eq:no_idea_for_label_name}
\binom{k}{2} \geq \frac{1}{2}k K_{p,\epsilon}^+ = \frac{(1+\epsilon)(1-p)k}{p},
\end{equation}
we can apply Lemma \ref{lem:Mk_results} (i) to obtain $c,C,L_2 \geq 0$ such that
\begin{align*}
\ex[C_{n,p}-1] &\leq K_{p,\epsilon}^+ +  C\sum_{k = K_{p,\epsilon}^++1}^n e^{-ck}
\end{align*}
for $K_{p,\epsilon}^+ \geq L_2$. Evaluating the sum on the right side we obtain an $L_3 > 0$ such that, when $K_{p,\epsilon}^+ \geq L_3$,
\begin{align*}
    C\sum_{k = K_{p,\epsilon}^++1}^n e^{-ck} \leq \delta K_{p,\epsilon}^+.
\end{align*}
Hence, for $K_{p,\epsilon}^+ \geq L_2 \vee L_3$ we have
\begin{equation*}
\limsup_{n \to \infty} \frac{\ex[C_{n,p}-1]}{K_{p,\epsilon}^+} \leq (1+\delta).
\end{equation*}
As before, we may increase $L_2$ or $L_3$ in order to obtain the bound,
\begin{equation}\label{eq:limsup_Cnp}
\limsup_{n \to \infty} \frac{2(1-p)\ex[C_{n,p}]}{p} \leq (1+\delta)^2(1+\epsilon).
\end{equation}
Given our freedom over the choice of $\delta$ and $\epsilon$, the first result of Lemma \ref{lem:Cnp_bounds} follows straightforwardly from combining (\ref{eq:liminf_Cnp}) with (\ref{eq:limsup_Cnp}).

The second result follows from a similar approach. Recalling (\ref{eq:no_idea_for_label_name}), we have
$$
\limsup_{n \to \infty}\prob\left(C_{n,p} > K_{p,\epsilon}^++1\right) \leq \limsup_{n \to \infty}\prob\left(\bigcup_{k=K_{p,\epsilon}^++1}^n \left\{M_{n-k} \geq \frac{(1+\epsilon)(1-p)k}{p}\right\} \right).
$$
Then, by Lemma \ref{lem:Mk_results} (i), we obtain $c',C',L_4 > 0$ such that, for $n \geq \ell \geq L_4$,
$$
\limsup_{n \to \infty}\prob\left(\bigcup_{k=K_{p,\epsilon}^++1}^n \left\{M_{n-k} \geq \frac{(1+\epsilon)(1-p)k}{p}\right\} \right) \leq C'e^{-c'K_{p,\epsilon}^+}.
$$
Since $K_{p,\epsilon}^+ \geq K_{p,\epsilon}^-$, there exists $L_5 > 0$ such that when $K_{p,\epsilon}^- \geq L_5$ we have $C'e^{-c'K_{p,\epsilon}^+} \leq \delta$, which establishes the claimed upper bound. For the lower bound we use (\ref{eq:lower_bound_for_Mk_in_Cnp_proof}) along with the aforementioned fact that $\lim_{n \to \infty} K_{p,\epsilon}^-n^{-1} = 0$ to get the existence of a constant $L_1 \geq 0$ such that, when $\liminf_{n \to \infty}K_{p,\epsilon}^- \geq L_1$,
\begin{align*}
\limsup_{n \to \infty}\prob(C_{n,p} \leq K_{p,\epsilon}^- + 1 ) &\leq \limsup_{n \to \infty} \prob\left(M_{n-K_{p,\epsilon}^-} \leq \binom{K_{p,\epsilon}^-}{2}\right) \\
&\leq \limsup_{n \to \infty} \left(\delta + \frac{1}{2\epsilon}\left(\frac{K_{p,\epsilon}^+-1}{n-1}\right)\right) = \delta. 
\end{align*}
Taking $K_{p,\epsilon}^- \geq L_1 \vee L_4 \vee L_5$ we obtain the desired result.
\end{proof}

Equipped with Lemma \ref{lem:Cnp_bounds}, Theorems \ref{thm:trees} and \ref{thm:trees_small_p} follow without too much extra effort. We restate Theorem \ref{thm:trees} for reference:

\begin{theorem*}
    There exists a family of random variables $(C_p : p \in (0,1))$ such that
    \begin{enumerate}
        \item $C_{n,p} \convdist C_p$ and $\ex[C_{n,p}] \to \ex[C_p]$ as $n \to \infty$ for any fixed $p \in (0,1)$; and
        \item $\frac{p}{2(1-p)}C_p \convprob 1$ and $\frac{p}{2(1-p)}\ex[C_p] \to 1$ as $p \to 0$.
    \end{enumerate}
\end{theorem*}

\begin{proof}[Proof of Theorem \ref{thm:trees}]
We only prove (i), as (ii) can be easily proven by simply combining (i) with the previous lemma. By applying the second result in Lemma \ref{lem:Cnp_bounds} we see that, for fixed p, the sequence $(C_{n,p})_{n=0}^\infty$ is a tight family of random variables. By Prokhorov's Theorem, there is some subsequence $(C_{n_k,p})_{k=0}^\infty$ and some random variable $C_p$ on $\N$ such that $C_{n_k,p} \convdist C_p$ as $k \to \infty$ \cite{billingsley2013convergence}. The monotonicity from Corollary \ref{cor:monotonicityofkingman} combined with the subsequential convergence implies that $C_{n,p} \convdist C_p$ as $n \to \infty$.

Since the sequence $(C_{n,p})_{n=1}^\infty$ is such that $C_{m,p} \preceq C_{n,p}$ for all $0 \leq m \leq n$ by Corollary \ref{cor:monotonicityofkingman}, it follows from the monotone convergence theorem that $\E[C_{n,p}] \to \E[C_p]$ as $n \to \infty$. To see this, for each $n \geq 0$ and $k \geq 1$, let $q_{n,k} = \prob(C_{n,p} \geq k)$ and set $q_k = \prob(C_p \geq k)$. Let $U \dist \unif[0,1]$ and define random variables $(X_n)_{n=0}^\infty$ and $X$ as follows:
$$
X_n = \sum_{k = 1}^\infty k\1_{\{ q_{n,k} \leq U < q_{n,k+1} \}}, \ X = \sum_{k = 1}^\infty k\1_{\{ q_{k} \leq U < q_{k+1} \}}.
$$
Note that, by definition, $X_n \dist C_{n,p}$ for all $n \geq 0$ and that $X \dist C_p$. By Corollary \ref{cor:monotonicityofkingman} it holds that $X_m(\omega) \leq X_n(\omega)$ for all $0 \leq m \leq n$. By the fact that $\prob(C_{n,p} \geq k) \to \prob(C_p \geq k)$, we have that $X_n(\omega) \leq X(\omega)$ for all $n \geq 0$. Moreover, $X_n \convas X$ as $n \to \infty$ since
$$
\left\{\omega \in \Omega : \lim_{n \to \infty} X_n(\omega) \ne X(\omega)\right\} \subseteq \{q_1,q_2,...\}.
$$
Now apply the monotone convergence theorem to conclude.
\end{proof}

Since $K_{p,\epsilon}^- \to \infty$ as $n \to \infty$ whenever $p \to 0$ as $n \to \infty$, Theorem \ref{thm:trees_small_p} follows directly from Lemma \ref{lem:Cnp_bounds}.

\section{Structural properties of $F(G_{n,p})$}\label{sec:discussion}

Many statistics of the trees in a Kingman forest of a $G_{n,p}$ can be determined by using the useful connection between $\king(K_n)$ and uniform random recursive trees, which we briefly introduce now. The \emph{uniform random recursive tree process} is an infinite sequence of random rooted trees $(T_n)_{n=1}^\infty$, where $T_1$ consists of a root labelled $1$, and $T_{n+1}$ is derived from $T_n$ by attaching a vertex labelled $n+1$ to a uniform vertex from $T_n$. The tree $T_n$ is called a uniform random recursive tree of size $n$. Much is known about the structural properties of $T_n$ as $n \to \infty$, including statistics like the height, max degree, and profile \cite{devroye1987branching,pittel1994note,devroye1995strong,dobrow1999total,goh2002limit,janson2005asymptotic,fuchs2006profiles,zhang2015number}. It turns out \cite{devroye1987branching,addario2018high,eslava2022depth} that a Kingman forest of $K_n$, upon re-labelling the vertices and edges in a way that we describe later, is distributed like $T_n$. Since labellings do not affect the structure of the trees, this connection can be leveraged to deduce information about label--independent properties of either model by studying the other.

The \emph{uniform random recursive forest process with $k$ trees} is a sequence of forests $(F_{n,k})_{n=k}^\infty$ defined recursively. First, $F_{k,k}$ is a graph with $k$ roots labelled $1 , ... , k$ and no edges. $F_{n+1,k}$ is derived from $F_{n,k}$ by adding an directed edge from a new vertex with the label $n+1$ to a uniformly chosen vertex in $F_{n,k}$. We write $(F_{n,k})_{n=k}^\infty \dist \urrf_k$ and $F_{n,k} \dist \urrf_k(n)$.

Let $f \in \cF_{n,n-k}$. Suppose that its roots are $x_1 \leq ... \leq x_k$. For all $i \in [n] \setminus \{x_1 , ... , x_k\}$, let $\ell_f(i)$ be the label of the unique edge that has $i$ as its tail. We define a new random labelling of the vertices $L_f:[n] \to [n]$ as follows:
    $$
    L_f(i) = 
    \begin{cases}
    j, \text{ if } i = x_j \\
    n-\ell_f(i) + 1, \text{ if } i \notin \{x_1 , ... , x_k\}
    \end{cases}
    .
    $$
Let $\Phi(f)$ be forest that is obtained from $f$ by removing the edge labellings, and relabelling the vertices by $L_f$.

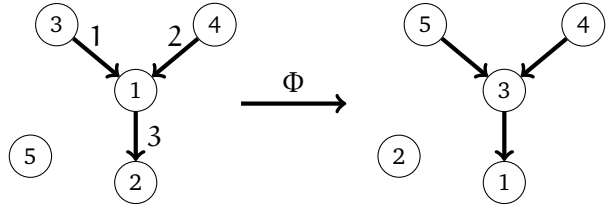
\begin{figure}[H]
    \centering
    \begin{tikzpicture}[scale=0.7]
    \begin{scope}[every node/.style={circle,draw}]
    \node[scale=0.8] (1') at (1,1.25) {1};
    \node[scale=0.8] (2') at (1,-0.5) {2};
    \node[scale=0.8] (3') at (-0.5,2.5) {3};
    \node[scale=0.8] (4') at (2.5,2.5) {4};
    \node[scale=0.8] (5') at (-1,0) {5};
    \end{scope}

    \draw[->,ultra thick] (1') -- (2') node[midway,right] {$3$};
    \draw[->,ultra thick] (4') -- (1') node[midway,above] {$2$};
    \draw[->,ultra thick] (3') -- (1') node[midway,above] {$1$};

    \begin{scope}[every node/.style={circle,draw}]
    \node[scale=0.8] (1'') at (8,1.25) {3};
    \node[scale=0.8] (2'') at (8,-0.5) {1};
    \node[scale=0.8] (3'') at (6.5,2.5) {5};
    \node[scale=0.8] (4'') at (9.5,2.5) {4};
    \node[scale=0.8] (5'') at (6,0) {2};
    \end{scope}

    \draw[->,ultra thick] (1'') -- (2'');
    \draw[->,ultra thick] (4'') -- (1'');
    \draw[->,ultra thick] (3'') -- (1'');

    \draw[->,ultra thick] (3,1) -- (5,1) node[midway,above] {$\Phi$};
    \end{tikzpicture}
    \caption{A forest in $\cF_{5,3}$ and its image under $\Phi$.}
    \label{fig:randomrecursiveforest}
\end{figure}

\begin{lemma}\label{lem:randomrecursiveforest}
Let $F \dist \unif(\cF_{n,n-k})$. Then, $\Phi(F) \dist \urrf_k(n)$.
\end{lemma}

\begin{proof}
Let $\cR_{n,k}$ be the set of all forests on $n$ vertices, with $k$ trees, whose trees are increasing. Simple inductive arguments on $k$ show that $|\cR_{n,k}| = \frac{(n-1)!}{(k-1)!}$, that $F_{n,k} \dist \unif(\cR_{n,k})$, and that $|\cF_{n,n-k}| = \frac{n!(n-1)!}{k!(k-1)!}$. From these observations we conclude that, to show the desired result, it suffices to show that $\Phi:\cF_{n,n-k} \to \cR_{n,k}$ is an $\frac{n!}{k!}$ to $1$ surjection.

Let $f \in \cR_{n,k}$. It is easy to see that $|\Phi^{-1}(f)| > 0$ by considering the forest obtained by labelling each edge, $(u,v)$, of $f$ by $n - u + 1$ and relabelling the non-root vertices arbitrarily in the set $[n] \setminus [k]$. Then, observe that applying two distinct permutations $\sigma,\tau$ to the vertex labels of any particular $f \in \Phi^{-1}(f_2)$ that satisfies $\sigma(x_1) \leq ... \leq \sigma(x_k)$ and $\tau(x_1) \leq ... \leq \tau(x_k)$ yields two distinct forests $f_\sigma$ and $f_\tau$ that both are in the set $\Phi^{-1}(f_2)$. From this observation, we get that $|\Phi^{-1}(f_2)| \geq \frac{n!}{k!}$ (the number of permutations that satisfy the described constraint).

Next, suppose that $f_2,f_2' \in \Phi^{-1}(f)$. Let $x_1 , ... , x_k$ and $x_1' , ... , x_k'$ be the roots of $f_2$ and $f_2'$ respectively. We define a function $\sigma:[n] \to [n]$ as follows. First, we set $\sigma(x_j) = x_j'$ for all $1 \leq j \leq k$. Then, for all $i \in [n] \setminus \{x_1 , ... , x_k\}$, we set $\sigma(i)$ to be the unique vertex $j$ in $f_2'$ such that $\ell_{f_2}(i) = \ell_{f_2'}(j)$. $\sigma$ is clearly a bijection and is clearly edge-label-preserving. If we can show that it is a graph isomorphism between $f_2$ and $f_2'$, then it follows that $|\Phi^{-1}(f)| \leq \frac{n!}{k!}$, and so $|\Phi^{-1}(f)| = \frac{n!}{k!}$.

By the symmetry of the two forests, to show that $\sigma$ is an isomorphism, it suffices to show that $(\sigma (u) , \sigma (v)) \in E(f_2')$ for all $(u,v) \in  E(f_2)$. By the definitions of $\sigma$ and $\Phi$, we have that $L_{f_2}(u) = L_{f_2'}(\sigma(u))$ and $L_{f_2}(v) = L_{f_2'}(\sigma(v))$. Since $\Phi(f_2) = f$, it holds that $(L_{f_2}(u) , L_{f_2}(v)) \in E(f)$, and so $(L_{f_2'}(\sigma(u)) , L_{f_2'}(\sigma(v))) \in E(f)$ as well. Since $L_{f_2'}$ is just a relabelling of the vertices, we conclude that $(\sigma(u) , \sigma(v)) \in E(f_2')$.

\end{proof}

Now fix $p \in (0,1)$. 
From Lemma \ref{lem:uniform} (i), for $n \in \N$ we have that, conditional upon $C_{n,p}$, $F(G_{n,p})$ is a uniform element of $\cF_{n,n-C_{n,p}}$. By Lemma \ref{lem:randomrecursiveforest}, a uniform element of $\cF_{n,n-C_{n,p}}$ has the same graph structure as a uniform random recursive forest with $C_{n,p}$ trees. We can use this fact to derive information about the structure of $F(G_{n,p})$. First we cover the sizes of the trees in the forest.

\begin{theorem}\label{thm:treesizes}
Fix $p \in (0,1)$ and let $X_n = (|S_1|,...,|S_{C_{n,p}}|)$ be the sizes of the trees in $F(G_{n,p})$. Then, $(\frac{1}{n}X_n,C_{n,p}) \convdist (X,C_{p})$ as $n \to \infty$, where $C_{p}$ is the random variable from Theorem \ref{thm:trees} and, conditional upon $C_p$, $X$ has a Dirichlet distribution with parameters $\alpha_1  = ... = \alpha_{C_{p}} = 1$.
\end{theorem}

\begin{proof}
For all $n \geq k \geq 0$, let $Y_{n}^{(k)} = (Y_{n,1}^{(k)},...,Y_{n,k}^{(k)})$ be distributed like the number of balls of each colour $1,...,k$ in a standard P\'{o}lya urn that is initialized with one ball of each colour after $(n-k)$ balls have been added to the system. Let $Z_k = (Z_{k,1},...,Z_{k,k})$ be a Dirichlet random variable with parameters $\alpha_1 = \cdots = \alpha_k = 1$. By applying Lemmas \ref{lem:uniform} and \ref{lem:randomrecursiveforest} it holds for any $ 0 \leq x_1,...,x_k \leq n$ that
    \begin{equation}\label{eq:polya_urn}
    \prob\left(|S_1| \geq x_1 , ... , |S_{k}| \geq x_k \ | \ C_{n,p} = k\right) = \prob(Y_{n,1}^{(k)} \geq x_1 , ... , Y_{n,k}^{(k)} \geq x_k).
    \end{equation}
To see this simply note that, at any step of the uniform random recursive forest process, the conditional probability that the process adds a vertex to a given tree is exactly proportional to the size of the tree, so the vector of tree sizes in a sample from $\urrf_k(n)$ is distributed as $Y_n^{(k)}$. It is a well-known result from the theory of P\'{o}lya urns (see e.g., \cite{pemantle2007survey} Theorem 2.1 or \cite{mahmoud2008polya} Theorem 3.2) that, for any $ 0 \leq x_1,...,x_k \leq 1$,
    $$
    \prob\left(\frac{1}{n}Y_{n,1}^{(k)} \geq x_1 , ... , \frac{1}{n}Y_{n,k}^{(k)} \geq x_k\right) \to \prob(Z_{k,1} \geq x_1 , ... , Z_{k,k} \geq x_k)
    $$
as $n \to \infty$. By combining this convergence with  Theorem \ref{thm:trees} and (\ref{eq:polya_urn}) we get,
    $$
    \prob\left(\left\{\frac{1}{n}|S_1| \geq x_1 , ... , \frac{1}{n}|S_k| \geq x_k \right\} \bigcap \left\{ C_{n,p} = k\right\}\right) \to \prob(Z_{k,1} \geq x_1 , ... , Z_{k,k} \geq x_k , C_{p} = k)
    $$
as $n \to \infty$.
\end{proof}

We finish this section by identifying the asymptotic height of $F(G_{n,p})$.

\begin{theorem}\label{thm:heightthm}
The following two points hold:
\begin{enumerate}
    \item There exists a constant $K > 0$ such that 
    \[
    |\ex[\height(F(G_{n,p}))] - e\log(n) + \frac{3}{2}\log\log(n)| \leq K 
    \]
    for all $n \geq 1$.
    \item We have, $\frac{\height(F(G_{n,p}))}{e\log(n)} \convprob 1$ as $n \to \infty$.
\end{enumerate}
\end{theorem}

\begin{proof}
Let $(T_m)_{m=1}^\infty$ be distributed as the uniform random recursive tree process. Then, it is immediate from the respective definitions that $([m] , E(T_m) \setminus \binom{[m]}{2})_{m=k}^\infty \dist \urrf_k$. This fact, in combination with Lemma \ref{lem:uniform} (i), implies that we can generate $F(G_{n,p})$ by sampling $T_n$ and $C_{n,p}$ independently, then deleting the first $C_{n,p}$ edges from $T_n$. Since the deletion of all $k-1$ edges in $\binom{[k]}{2}$ from a uniform random recursive tree can at most reduce the height by $k-1$, this coupling gives us, for all $x > 0$,
\begin{equation}\label{eq:stochdom_urrt}
\prob(\height(T_n) - C_{n,p} \geq x) \leq \prob\left( \height(F(G_{n,p})) \geq x \right) \leq \prob\left( \height(T_n) \geq x \right)\, .
\end{equation}
By Corollary 1.3 of \cite{addario2013poisson}, it holds that
\[
\Lambda := \sup_{n \geq 1}\ex \ \Big|\height(T_n) - e\log(n) + \frac{3}{2}\log\log(n)\Big|  < \infty.
\] 
Also, from Theorem \ref{thm:trees}, we have that $\ex [C_{n,p}] \to \ex [C_p]< \infty$ as $n \to \infty$. Combining these two facts with (\ref{eq:stochdom_urrt}) we get,
$$
\sup_{n \ge 1}
\left| \ex\Big[\height\big(F(G_{n,p})\big)\Big] - e\log(n) + \frac{3}{2}\log\log(n)\right| \leq \sup_{n \ge 1}\ex[C_{n,p}] + \Lambda < \infty,
$$
proving the first result. For the second result, first note that, since $\Lambda < \infty$, Markov's inequality implies that $\frac{\height(T_n)}{e\log(n)} \convprob 1$ as $n \to \infty$. Then, since $\ex[C_p] < \infty$, we have that $\frac{C_{n,p}}{e\log(n)} \convprob 0$ as $n \to \infty$. Combining these two convergences with (\ref{eq:stochdom_urrt}) completes the proof of the second result.
\end{proof}

\subsection*{Computing other statistics in $F(G_{n,p})$}

    The usefulness of the coupling from the proof of Theorem \ref{thm:heightthm}, where we generate $F(G_{n,p})$ from independently sampled $C_{n,p}$ and $T_n$, is not limited in its usage to only discussion of the height. For example, this fact almost implies that, for any $i \in [n]$, $\deg_{F(G_{n,p})}(i)$ can be coupled with $\deg_{T_n}(i)$ so that $|\deg_{F(G_{n,p})}(i) - \deg_{T_n}(i)| \leq 1$. If we take $i = i(n) \to \infty$ as $n \to \infty$, we even have that $\prob(\deg_{F(G_{n,p})}(i)  \ne \deg_{T_n}(i)) \to 0$ as $n \to \infty$. From this we could derive a variety of results concerning the degrees in Kingman forests of $G_{n,p}$ with almost no extra effort. More generally, one can compute almost any statistic of interest that is understood for uniform random recursive trees by leveraging the fact that $([m] , E(T_m) \setminus \binom{[m]}{2})_{m=k}^\infty \dist \urrf_k$.

\section{Quantitative results on $M_k$}\label{sec:bound}

In this section we prove the results on $M_k$ in Lemma \ref{lem:Mk_results} that we used in the proof of our main results. We require the use of many fairly standard tail bounds for familiar collections of random variables in our analysis of $M_k$.

\begin{lemma}\label{lem:manybounds}
    Let $0 < \delta < 1$. Let $(X_k)_{k = 1}^n$ be an independent collection of random variables with $X_k \dist \hyper(d_k,m_k,n_k)$ for some $d_k , m_k \leq n_k$ and set $X = \sum_{k=1}^n X_k$ and $\mu = \ex[X] = \sum_{k=1}^n \frac{d_km_k}{n_k}$. Then, 
        $$
        \prob\left(\left|X-\mu\right| \geq \delta\mu\right) \leq 2\exp\left(-\frac{\delta^2 \mu}{3}\right).
        $$
    Let $X \dist \negbin(r,p)$. Then,
        $$
        \prob\left(X \geq \frac{(1+\delta) r (1-p)}{p} \right) \leq \exp\left( -\frac{((1-p)\delta)^2r}{6}\right),
        $$
    and
        $$
        \prob\left( X \leq (1-\delta)\frac{r(1-p)}{p} \right) \leq \exp\left( -\frac{((1-p)\delta)^2r}{3(1-\delta(1-p))} \right).
        $$
\end{lemma}

\begin{proof}

The first bound follows from an extension of Hoeffding's inequality to the setting of sampling without replacement \cite[Theorems 2 and 4]{hoeffding1994probability}.

The second and third inequalities follow from the close relationship between binomial and negative binomial random variables. For a $\negbin(r,p)$ random variable to be at least $k$, we need to observe at most $r$ successes from $r+k$ independent $\ber(p)$ trials. Hence,
    $$
    \prob\left( X \geq (1+\delta)\frac{ r (1-p)}{p} \right) = \prob\left( \bin\left(r+(1+\delta)\mu , p \right) \leq r \right),
    $$
where $\mu = \frac{r(1-p)}{p}$. Using a Chernoff bound gives
    $$
    \prob\left( X \geq (1+\delta)\frac{ r (1-p)}{p} \right) \leq \exp\left( -\frac{1}{3}\left(1 - \frac{r}{\mu^+}\right)^2\mu^+ \right),
    $$
where 
    $$
    \mu^+ := ((1+\delta)\mu+r)p = (1+\delta)r(1-p) + rp = (1+\delta (1-p))r.
    $$ 
From here one can simplify the expression in a straightforward way to derive the final result:
    \begin{align*}
    \exp\left( -\frac{1}{3}\left(1 - \frac{r}{\mu^+}\right)^2\mu^+ \right) &= \exp\left( -\frac{1}{3}\left((1+(1-p)\delta)r - 2r + \frac{r}{(1+(1-p)\delta)}\right) \right) \\
    &= \exp\left( -\frac{1}{3}\left(\frac{((1-p)\delta)^2r}{1+(1-p)\delta}\right) \right) \\
    &\leq \exp\left(-\frac{((1-p)\delta)^2r}{6}\right).
    \end{align*}
The corresponding lower bound is derived in an almost identical fashion, so we omit the proof. 
\end{proof}

Using the bounds from Lemma \ref{lem:manybounds} we can prove (i) in Lemma \ref{lem:Mk_results}, which we restate in the following lemma.

\begin{lemma}\label{lem:upper_tail_bound}
    For any $\eta,\epsilon \in (0,1)$, there exist $C,L,c > 0$ such that the following holds. Fix $p \in (0,\eta)$ and integers $n$ and $\ell$ such that $n \geq \ell \geq L \vee K_{p,\epsilon}^+$. Then,
        $$
        \prob\left( \bigcup_{k = 0}^{n-\ell} \left\{ M_k \geq (1+\epsilon)\frac{(1-p)(n-k)}{p} \right\} \right) \leq Ce^{-c\ell}.
        $$
\end{lemma}

\begin{proof}

Let
    $$
    E = \bigcup_{k = 0}^{n-\ell} \left\{ M_k \geq (1+\epsilon)\frac{(1-p)(n-k)}{p} \right\}.
    $$ For all $0 \leq k \leq n$, set
    $$
    I_k = \left[\left(1 + \frac{\epsilon}{2}\right)\frac{(1-p)(n-k)}{p}, \left(1 + \epsilon\right)\frac{(1-p)(n-k)}{p}\right].
    $$ 
    For each $1 \leq k \leq n - \ell$ and $1 \leq j \leq n-\ell-k$, let $A_{k,j}$ denote the event that the following three conditions (i)-(iii) hold:
    \begin{enumerate}
        \item $M_{i} \in I_i$ for $k < i < k+j$,
        \item $M_{k + j} \geq \left(1 + \epsilon\right)\frac{(1-p)(n-k-j)}{p} = \sup I_{k+j}$, and
        \item $M_{k} \leq \left(1 + \frac{\epsilon}{2}\right)\frac{(1-p)(n-k)}{p} = \inf I_{k}$.
    \end{enumerate}
    Since $M_0 = 0$, if $E$ occurs then there must be some $1 \leq k \leq n-\ell$ and $1 \leq j \leq n-\ell - k$ such that $A_{k,j}$ occurs. Set 
        $$
        \Delta_{k,j} := \sup I_{k+j}-\inf I_k = \frac{\epsilon(1-p)(n-k)}{2p} - \frac{(1+\epsilon)(1-p)j}{p},
        $$
    and $T_k = \frac{\epsilon(n-k)}{4(1+\epsilon)}$. We split the bounding of $\prob(E)$ into two cases with the union bound,
        \begin{align}\label{eq:unionboundlemma3.2}
        \prob(E) \leq \underbrace{\sum_{k=1}^{n-\ell}\sum_{j = 1}^{T_k} \prob(A_{k,j})}_{:=(I)} + \underbrace{\sum_{k=1}^{n-\ell}\sum_{j = T_k+1}^{n-\ell-k} \prob(A_{k,j})}_{:=(II)}.
        \end{align}
    We shall bound $(I) $ and $(II)$ separately, beginning with $(I)$. One can show by a brief computation that $\Delta_{k,j} \geq \frac{(1+\epsilon)(1-p)T_k}{p}$ for $(k,j) \in \{ (i_1,i_2) : 1 \leq i_1 \leq n-\ell, \ 1 \leq i_2 \leq T_k \}$. 
\noindent Indeed, for $1 \leq k \leq n-\ell$ and $1 \leq j \leq T_k$, the map $j \mapsto \Delta_{k,j}$ is decreasing, so $ \Delta_{k,j} \geq \Delta_{k,T_k}$. Using the definitions of $\Delta_{k,j}$ and 
$T_k$, 
we then compute
\begin{align*}
\Delta_{k,j}
&\geq \frac{\epsilon(1-p)(n-k)}{2p} - \frac{(1+\epsilon)(1-p)T_k}{p} \\
&= \frac{1-p}{p}\left(\frac{\epsilon(n-k)}{2} - (1+\epsilon)T_k\right) \\
&= \frac{1-p}{p}\left(\frac{\epsilon(n-k)}{2} - (1+\epsilon)\cdot\frac{\epsilon(n-k)}{4(1+\epsilon)}\right) \\
&= \frac{1-p}{p}\cdot \frac{\epsilon(n-k)}{4}
= \frac{(1+\epsilon)(1-p)T_k}{p}.
\end{align*}
    
    By using the characterization of $(M_k)_{k=0}^{n-2}$ given in Lemma~\ref{lem:recursion}, we have that $M_{k+j} - M_j$ is stochastically dominated by a sum of $j$ independent $\geo(p)$--distributed random variables, which is $\negbin(j,p)$--distributed. Using this along with the negative binomial bound from Lemma~\ref{lem:manybounds} we get that, for $1 \leq j \leq T_k$,
        \begin{align*}
        \prob(A_{k,j}) \leq \prob\left( \negbin(T_k,p) \geq \frac{(1+\epsilon)(1-p)T_k}{p} \right) \leq 2\exp\left( -\frac{(1-p)^2\epsilon^2T_k}{6} \right).
        \end{align*}
    Since $p \leq \eta$, we may compress all of the constants into some $c_1 = c_1(\epsilon,\eta) > 0$ to get
    \begin{align*}
        (I) &\leq \frac{\epsilon}{2(1+\epsilon)}\sum_{k=1}^{n-\ell}(n-k)\exp\left(-c_1(n-k)\right) \\
        &\leq \frac{\epsilon}{2(1+\epsilon)}\sum_{k=\ell}^\infty k\exp\left( -c_1k\right).
    \end{align*}
    Doing a routine comparison of the above sum with an integral we can obtain a second constant $c_2 = c_2(\epsilon,\eta) > 0$ such that
    \begin{equation}\label{eq:Final_I}
        (I) \leq c_2\ell e^{-c_1 \ell}.
    \end{equation}
    To bound $(II)$, we need to consider the edges that are removed during the complement process as well. Essential to proceeding computations is the following claim that bounds $\prob(A_{k,j})$ by the probability of an event concerning sums of i.i.d.\ random variables.

    \begin{claim}
    For $(k,j) \in \{ (i_1,i_2) : 1 \leq i_1 \leq n-\ell, \ T_{k}+1 \leq i_2 \leq n-\ell-k \}$, we have that
    \begin{equation}\label{eq:Concentration_II}
        \prob(A_{k,j}) \leq \prob\left(X^*_{k + j - 1} + \sum_{i=0}^{j-2} (X^*_{k+i} - Y^*_{k+i}) \geq \Delta_{k,j}\right),
    \end{equation}
    where all the random variables $(X_{k + i}^*,Y_{k+i}^*)_{0 \leq i \leq j-1}$ are independent, with $X_{k+i}^* \dist \geo(p)$ and 
    $$
    Y_{k+i}^* \dist \hyper\left( n-k-i-2 , \left\lfloor\frac{(1 + \epsilon/2)(1-p)(n-k-i)}{p}\right\rfloor, \binom{n-k-i}{2} \right).
    $$
    \end{claim}

    \begin{proof}
    Let $\mathcal{F}$ be the sigma algebra generated by the whole edge reveal process. First, we note by the definition of $A_{k,j}$ that
    \begin{align*}
        \prob(A_{k,j}) \leq \prob\left( \left(\bigcap_{i=1}^{j-1}\big\{ M_i \in I_i \big\}\right) \bigcap \left\{X_{k + j - 1} + \sum_{i=0}^{j-2} (X_{k+i} - Y_{k+i}) \geq \Delta_{k,j}\right\}\right).
    \end{align*}
    From here, we complete the proof with a direct coupling. We define, for all $0 \leq i \leq j-1$ conditionally given $X_{k+i}$ and $M_{k+i}$,
    $$
    X_{k+i}^* = X_{k+i} + Z_{k+i}\1_{\{ M_{k+i} + X_{k+i} = \binom{n-k-i}{2}\}},
    $$
    where $(Z_{k+i} : 0 \leq i \leq j-1)$ is a collection of independent $\geo(p)$ random variables that is also independent of $\mathcal{F}$. By the memoryless property, we have that $X_{k+i}^* \dist \geo(p)$. Recall that, when $\tau_{k+i} < \infty$, we have that $Y_{k+i} = \deg_{G^*_{\tau_{k+i}-1}}(u_{k+i})$. For all $0 \leq i \leq j-2$, if $\tau_{k+i} < \infty$ and $G^*_{G^*_{\tau_{k+i}-1}}$ has more than $\lfloor \frac{(1+\epsilon/2)(1-p)(n-k-i)}{p}\rfloor$ edges, let $Y_{k+i}^* = \deg_{H_{k+i}}(u_{k+i})$, where $H_{k+i}$ is a uniformly chosen subgraph of $G^*_{G^*_{\tau_{k+i}-1}}$ with $\lfloor \frac{(1+\epsilon/2)(1-p)(n-k-i)}{p}\rfloor$ edges. Otherwise, we just set $Y_{k+i}^*$ to be a hypergeometric random variable with our desired distribution, independent of $\mathcal{F}$. By construction, $X_{k+i}^*$ and $Y_{k+i}^*$ are independent, and have the correct distribution. Finally, since $X_{k+i} \leq X_{k+i}^*$ for all $0 \leq i \leq j-1$, and since $Y_{k+i} \geq Y^*_{k+i}$ on the event that $M_{k+i} + X_{k+i} \leq \binom{n-k-i}{2}$ (which is a subset of the event that $M_{k+i+1} \in I_{k+i+1}$), we have that
    $$
     \left(\bigcap_{i=1}^{j-1}\big\{ M_i \in I_i \big\}\right) \bigcap \left\{X_{k + j - 1} + \sum_{i=0}^{j-2} (X_{k+i} - Y_{k+i}) \geq \Delta_{k,j}\right\}
    $$
    is a subset of
    $$
    \left\{X_{k + j - 1}^* + \sum_{i=0}^{j-2} (X_{k+i}^* - Y_{k+i}^*) \geq \Delta_{k,j}\right\},
    $$
    which is enough to complete the proof.
    \end{proof}
    With the claim proven, we can begin to work on bounding $(II)$, by bounding the expression in the right side of (\ref{eq:Concentration_II}). Set, for all $1 \leq k \leq n-\ell$ and $T_k+1 \leq j \leq n-\ell-k$,
\begin{align*}
\mu_{k,j}
& = \ex\!\left[\,X^*_{k+j-1}+\sum_{i=0}^{j-2}X^*_{k+i}\right]
= \frac{j(1-p)}{p},\\
\nu_{k,j}
& = \ex\!\left[\sum_{i=0}^{j-2}Y^*_{k+i}\right]
= \sum_{i=0}^{j-2}\left(1+\frac{\epsilon}{2}\right)\frac{2(1-p)(n-k-i-2)}{p(n-k-i-1)}.
\end{align*}
From $(n-k-i-2)/(n-k-i-1)=1-\frac{1}{\,n-k-i-1\,}$ we obtain
\begin{align*}
\mu_{k,j}-\nu_{k,j}
&= \frac{j(1-p)}{p}
 -\left(1+\frac{\epsilon}{2}\right)\frac{2(1-p)}{p}
    \sum_{i=0}^{j-2}\!\left(1-\frac{1}{\,n-k-i-1\,}\right)\\
&\le \frac{1-p}{p}
   \Big(-\big((1+\epsilon)j-(2+\epsilon)\big)
       + 4\sum_{r=n-k-j+1}^{n-k}\frac{1}{r}\Big) \\
&\le \frac{1-p}{p}\Big(-\big((1+\epsilon)j-(2+\epsilon)\big)+4\log(n-k)\Big),
\end{align*}
where in the last step we used the crude bound
$\sum_{r=n-k-j+1}^{n-k}\! r^{-1}\le \log(n-k)$.

Fix
$
\delta=\epsilon/20\in(0,1)
$.
Since $\epsilon < 1$ we have $\mu_{k,j}+\nu_{k,j} \le 5j(1-p)/p$, so 
\begin{align*}
(1+\delta)\mu_{k,j}-(1-\delta)\nu_{k,j}
&\le \frac{1-p}{p}\Big(-\big((1+\epsilon-5\delta)j-(2+\epsilon)\big)+\log(n-k)\Big).
\end{align*}
Using that $j\le n-k-\ell$ and $5\delta=\epsilon/4$, as well as the definition of $\Delta_{k,j}$, we obtain 
\begin{align*}
(1+\delta)\mu_{k,j}-(1-\delta)\nu_{k,j}-\Delta_{k,j}
&\le \frac{1-p}{p}\left(5\delta j + 2+\epsilon - \frac{\epsilon}{2}(n-k) + \log(n-k)\right)\\
&\le \frac{1-p}{p}\left(3-5\delta\,\ell - \frac{\epsilon}{4}(n-k) + \log(n-k)\right).
\end{align*}
Since the linear term dominates the logarithm, there exists
$L=L(\epsilon,\eta)>0$ such that for all $s\ge L$,
$-\tfrac{\epsilon}{4}s+\log s+3\le -\tfrac{\epsilon}{8}s$.
Hence, for all $n-k\ge L$ and all $T_k\le j\le n-k-\ell$,
\begin{equation*}
(1+\delta)\mu_{k,j}-(1-\delta)\nu_{k,j}
\le
\Delta_{k,j}
-\frac{\epsilon}{8}\frac{1-p}{p}(n-k)
-\frac{5\delta(1-p)}{p}\,\ell.
\end{equation*}
(Since $n-k \ge \ell$, for the above bound to hold it suffices that $\ell \ge L$ and $T_k \le j \le n-k-\ell$.)
Consequently,
\begin{align*}
& \Big\{\sum_{i=0}^{j-1}X^*_{k+i}\le (1+\delta)\mu_{k,j}\Big\}
\bigcap
\Big\{\sum_{i=0}^{j-2}Y^*_{k+i}\ge (1-\delta)\nu_{k,j}\Big\}\\
& \subseteq
\Big\{X^*_{k+j-1}+\sum_{i=0}^{j-2}(X^*_{k+i}-Y^*_{k+i})<\Delta_{k,j}\Big\}.
\end{align*}
Thus, by \eqref{eq:Concentration_II} and a union bound, if $\ell \ge L$ and $T_k \le j \le n-k-\ell$ then 
\[
\prob(A_{k,j})
\le
\prob\!\left(\sum_{i=0}^{j-1}X^*_{k+i}\ge (1+\delta)\mu_{k,j}\right)
+
\prob\!\left(\sum_{i=0}^{j-2}Y^*_{k+i}\le (1-\delta)\nu_{k,j}\right).
\]

Lemma~\ref{lem:manybounds} yields a constant $c_3=c_3(\epsilon,\eta)>0$ such that 
\begin{equation}\label{eq:negbinbound_II_new}
\prob\!\left(\sum_{i=0}^{j-1}X^*_{k+i}\ge (1+\delta)\mu_{k,j}\right)
\le \exp\!\left(-\frac{((1-p)\delta)^2}{6}\,j\right)
\le \exp\!\left(-c_3\,(n-k)\right),
\end{equation}
since $j\ge T_k=\frac{\epsilon}{4(1+\epsilon)}(n-k)$ in case $(II)$.
Also, applying the first inequality in Lemma~\ref{lem:manybounds} to the family $(Y^*_{k+i})_{i=0}^{j-2}$, and using the crude bound
\[
\nu_{k,j}
=\sum_{i=0}^{j-2}\left(1+\frac{\epsilon}{2}\right)\frac{2(1-p)(n-k-i-2)}{p(n-k-i-1)}
\ge
\left(1+\frac{\epsilon}{2}\right)\frac{1-p}{p}\,(j-1),
\]
since $n-k \ge \ell$, 
we obtain $c_4=c_4(\epsilon,\eta)>0$ such that, for all $T_k\le j\le n-k-\ell$,
\begin{equation}\label{eq:hypergeobound_II_new}
\prob\!\left(\sum_{i=0}^{j-2}Y^*_{k+i}\le (1-\delta)\nu_{k,j}\right)
\le 2\exp\!\left(-\frac{\delta^2}{3}\,\nu_{k,j}\right)
\le 2\exp\!\left(-c_4\,(n-k)\right).
\end{equation}

\medskip\noindent\emph{Putting the pieces together.}
Combining \eqref{eq:negbinbound_II_new} and \eqref{eq:hypergeobound_II_new} gives
\begin{equation}\label{eq:combined_II}
\prob(A_{k,j})
\le \exp\!\left(-c_3(n-k)\right)+2\exp\!\left(-c_4(n-k)\right),
\end{equation}
for all $\ell \ge L\vee K_{p,\epsilon}^+$ and $T_k\le j\le n-k-\ell$.
Hence, summing over $j$ and $k$ in the contribution $(II)$ from \eqref{eq:unionboundlemma3.2}, we obtain
\begin{align*}
(II)
&\le \sum_{k=1}^{n-\ell}\sum_{j=T_k}^{n-\ell-k}\prob(A_{k,j})
 \le \sum_{k=\ell}^{\infty}k\,e^{-c_3k}
      + 2\sum_{k=\ell}^{\infty}k\,e^{-c_4k}.
\end{align*}
Using this together with \eqref{eq:Final_I} in (\ref{eq:unionboundlemma3.2}), the result follows. 
\end{proof}

Since $M_0 = 0$, proving a lower bound matching the one provided in Lemma \ref{lem:upper_tail_bound} requires a different approach. Specifically, we are faced with the new problem of verifying that $M_k$ ever gets within $\frac{\epsilon(1-p)(n-k)}{p}$ of the desired value of $\frac{(1-p)(n-k)}{p}$. The next two lemmas combine to show this.

\begin{lemma}\label{lem:expectation_bounds_for_Mk}
For all $\epsilon,\eta \in (0,1)$ there exists $L >0$ such that, for all $p \in (0,\eta)$ and integers $\ell,n$ with $n\ge \ell\ge L\vee K_{p,\epsilon}^+$,
\[
\ex[M_{n-\ell}]\ \ge\  \frac{(1-\epsilon)(1-p)}{p}\,\ell\Big(1-\frac{\ell-1}{n-1}\Big).
\]
\end{lemma}
\begin{proof}
Let $\cF_k=\sigma\!\big((X_j,Y_j)_{j<k}\big)$ and recall that $M_{k+1}=M_k+X_k-Y_k$. As in the exposure process,
\[
\ex\!\left[Y_k\,\middle|\,\cF_k,X_k\right]\ \le\ \frac{2}{n-k}\,(M_k+X_k),
\]
and $X_k$ is a $\geo(p)$ truncated at $B_k:=\binom{n-k}{2}-M_k$, so
\[
\ex[X_k\mid\cF_k] = \frac{1-p}{p}\,\big(1-p^{B_k}\big).
\]
Therefore, for $0\le k\le n-2$,
\begin{align}
\ex[M_{k+1}\mid\cF_k]
&= M_k+\ex\!\left[X_k-\ex\!\left[Y_k\,\middle|\,\cF_k,X_k\right]\middle|\,\cF_k\right] \notag\\
&\ge \Big(1-\frac{2}{n-k}\Big)M_k+\frac{1-p}{p}\Big(\,1-\frac{2}{n-k}-p^{B_k}\Big). \label{eq:cond-step}
\end{align}

Choose $L_1$ so that $\frac{2}{n-k}\le \epsilon/2$ whenever $n-k\ge L_1$. Next, with
\[
t_k:=\Big(1+\frac{\epsilon}{2}\Big)\frac{(1-p)(n-k)}{p},
\]
we have
\[
\ex[p^{B_k}]\ =\ \ex[p^{B_k}\mathbf{1}_{\{M_k\ge t_k\}}]+\ex[p^{B_k}\mathbf{1}_{\{M_k<t_k\}}]
\ \le\ \prob(M_k\ge t_k)+p^{{n-k \choose 2}-t_k}.
\]
By Lemma~\ref{lem:upper_tail_bound}, there exist $C,c,L_2>0$ such that $\prob(M_k\ge t_k)\le Ce^{-c(n-k)}$ for $n-k\ge L_2 \vee K_{p,\epsilon}^+$. Moreover,
\[
\binom{n-k}{2}-t_k=\binom{n-k}{2}-\Big(1+\frac{\epsilon}{2}\Big)\frac{(1-p)(n-k)}{p}
\ \ge\ \frac{\epsilon(1-p)}{2p}\,(n-k)
\]
whenever $n-k\ge \frac{2(1+\epsilon)(1-p)}{p}+1$, hence $p^{B_k-t_k}\le e^{-c'(n-k)}$ (since $p\le\eta<1$). Thus there exists $L_3$ such that
\[
\ex[p^{B_k}] \le \epsilon/2\qquad\text{whenever } n-k\ge L_3\vee \frac{2(1+\epsilon)(1-p)}{p}.\]
Taking expectations in \eqref{eq:cond-step} and using the two bounds above, we obtain for all such $k$,
\begin{equation}\label{eq:main-rec}
\ex[M_{k+1}]
\ \ge\ \Big(1-\frac{2}{n-k}\Big)\ex[M_k]+\frac{(1-\epsilon)(1-p)}{p}.
\end{equation}

Set $c:=\frac{(1-\epsilon)(1-p)}{p}$ and $b_k:=1-\frac{2}{n-k}$. A particular solution of the recurrence $a_{k+1}=b_ka_k+c$ is $a_k^*:=c(n-k)$.  Writing $d_k:=\ex[M_k]-a_k^*$, \eqref{eq:main-rec} gives $d_{k+1}\ge b_k d_k$. With $M_0=0$ we have $d_0=-cn$, and therefore
\[
\ex[M_k]\ \ge\ c(n-k)-c n\prod_{i=0}^{k-1}\Big(1-\frac{2}{n-i}\Big).
\]
For $k=n-\ell$ this product telescopes to yield 
\[
\prod_{i=0}^{n-\ell-1}\Big(1-\frac{2}{n-i}\Big)
=\prod_{j=\ell+1}^{n}\frac{j-2}{j}
=\frac{\ell(\ell-1)}{n(n-1)}.
\]
Hence, for all $n\ge \ell\ge L\vee K_{p,\epsilon}^+$ with $L:=L_1\vee L_3$,
\[
\ex[M_{n-\ell}]
\ \ge\ c\Big(\ell-\frac{\ell(\ell-1)}{n-1}\Big)
=\frac{(1-\epsilon)(1-p)}{p}\,\ell\Big(1-\frac{\ell-1}{n-1}\Big),
\]
which is the stated bound. 
\end{proof}

Combining the previous lemma with the exponential tail bound from Lemma \ref{lem:upper_tail_bound} we can prove that $M_k$ will cross above $\frac{(1-\epsilon)(1-p)(n-k)}{p}$ before the edge reveal process terminates.

\begin{lemma}\label{mnlowerbound}
    Let $\eta,\epsilon,\delta \in (0,1)$. There exists $L \geq 0$ such that, for any $n \geq 0$ and any $p \in (0,\eta)$ such that $ K_{p,\epsilon}^+ \geq L$, we have,
    $$
        \prob\left(M_{n-K_{p,\epsilon}^+} \leq \frac{(1-\epsilon)(1-p)K_{p,\epsilon}^+}{p}\right) \leq \delta + \frac{1}{2\epsilon}\left(\frac{K_{p,\epsilon}^+ - 1}{n-1}\right).
    $$
\end{lemma}

In the proof, we use the following basic probability fact.

\begin{lemma}\label{lem:basic_bound}
    Let $X$ be a random variable with $\ex[X] \in \R$, and let $a,b \in \R$. Then,
    $$
    \prob(X \leq a) \leq \frac{1}{b-a}\Big( b + \ex[X  :  X \geq b] - \ex[X] \Big)
    $$
\end{lemma}

\begin{proof}
    By the definition of the expectation we have,
    $$
    \ex[X] \leq a\prob(X \leq a) + b\prob(a < X < b) + \ex[X : X \geq b].
    $$
    Upper bounding $\prob(a < X < b)$ with $1 - \prob(X \leq a)$ and then re-arranging terms gives the desired result.
\end{proof}

\begin{proof}[Proof of Lemma~\ref{mnlowerbound}]
    Let
    $$
    a = \frac{(1-\epsilon)(1-p)K_{p,\epsilon}^+}{p}, \ \ b = \frac{(1+\epsilon\delta)(1-p)K_{p,\epsilon}^+}{p}.
    $$
    Our goal is to apply Lemma~\ref{lem:basic_bound} to obtain our desired result, and towards that we need to bound $\ex[M_{n-K_{p,\epsilon}^+}]$ and $\ex[M_{n-K_{p,\epsilon}^+} : M_{n-K_{p,\epsilon}^+} \geq b]$. By Lemma~\ref{lem:expectation_bounds_for_Mk} we have an $L_1 > 0$ such that, when $p$ is such that $K_{p,\epsilon}^+ \geq L_1$,
    \begin{equation}\label{eq:mnlowerbound_applying_expectation}
    \E[M_{n-K_{p,\epsilon}^+}] \geq \frac{\left(1 - \frac{\epsilon\delta}{2}\right)(1-p)}{p}K_{p,\epsilon}^+\left( 1 - \frac{K_{p,\epsilon}^+ -1}{n-1} \right)
    \end{equation}
    By Lemma \ref{lem:upper_tail_bound} there exists $c,C,L_1$ such that 
    \begin{equation}\label{eq:Mk_contradiction_2}
    \prob\left( M_{n-\ell} \geq \left(1 + \frac{\epsilon\delta}{2}\right)\frac{(1-p)\ell}{p} \right) \leq Ce^{-c\ell}.
    \end{equation}
    whenever $\ell \geq L_2 \vee K_{p,\epsilon\delta/2}^+$.
    Using the fact that $M_{n-K_{p,\epsilon}^+}$ can only take values in $\{\binom{k}{2} : k \geq K_{p,\epsilon}^+\}$ on the event that $M_{n-K_{p,\epsilon}^+} \geq \binom{K_{p,\epsilon}^+}{2}$, we have,
    \begin{align}
        &\ex\left[M_{n-K_{p,\epsilon}^+} \1_{\left\{M_{n-K_{p,\epsilon}^+} \geq b\right\}}\right] \notag\\
        \leq& \sum_{j = b}^{\binom{K_{p,\epsilon}^+}{2}} j\prob\left(M_{n-K_{p,\epsilon}^+} = j \right) + \sum_{j = K_{p,\epsilon}^++1}^n \binom{j}{2}\prob\left(M_{n-K_{p,\epsilon}^+} = \binom{j}{2}\right). \label{eq:mnlowerbound_proof_bound_one}
    \end{align}
    The event $\{ M_{n-K_{p,\epsilon}^+} = \binom{j}{2} \}$ is exactly the event that the edge reveal process terminates between the $n-j$ coalescing time and $n-j-1$ coalescing time. Then, since 
    $$
    \binom{j}{2} \geq \left(1 + \frac{\epsilon\delta}{2}\right)\frac{(1-p)j}{p}
    $$
    for all $j \geq K_{p,\epsilon\delta/2}^+ + 1$, we can bound each of the terms in the second sum in (\ref{eq:mnlowerbound_proof_bound_one}) to obtain,
    \begin{align*}
        &\ex\left[M_{n-K_{p,\epsilon}^+} \1_{\left\{M_{n-K_{p,\epsilon}^+} \geq b\right\}}\right] \\
        \leq& \binom{K_{p,\epsilon}^+}{2}\prob\left(M_{n-K_{p,\epsilon}^+} \geq b \right) + \sum_{j = K_{p,\epsilon}^++1}^n \binom{j}{2}\prob\left(M_{n-j} \geq \left( 1 + \frac{\epsilon\delta}{2} \right)\frac{(1-p)j}{p}\right) \\
        \leq& \sum_{j = K_{p,\epsilon}^+}^n \binom{j}{2}\prob\left(M_{n-j} \geq \left( 1 + \frac{\epsilon\delta}{2} \right)\frac{(1-p)j}{p}\right).
    \end{align*}
    Then, using (\ref{eq:Mk_contradiction_2}) we get,
    \begin{align*}
        \ex\left[M_{n-K_{p,\epsilon}^+} \1_{\left\{M_{n-K_{p,\epsilon}^+} \geq b\right\}}\right] \leq C\sum_{j = K_{p,\epsilon}^+}^\infty j^2e^{-cj}.
    \end{align*}
    Since the upper bound above tends to zero as $K_{p,\epsilon}^+ \to \infty$, by potentially increasing $L_2$ if needed we obtain that, when $K_{p,\epsilon}^+ \geq L_2$,
    \begin{equation}\label{eq:Mk_contradiction_last_one_for_real}
        \ex\left[M_{n-K_{p,\epsilon}^+} \1_{\left\{M_{n-K_{p,\epsilon}^+} \geq b\right\}}\right] \leq \frac{\epsilon\delta}{2}\frac{(1-p)K_{p,\epsilon}^+}{p}.
    \end{equation}
    Applying Lemma~\ref{lem:basic_bound} with the bounds (\ref{eq:mnlowerbound_applying_expectation}) and (\ref{eq:Mk_contradiction_last_one_for_real}) we get, whenever $K_{p,\epsilon}^+ \geq L_1 \vee L_2$,
    $$
    \prob(M_{n-K_{p,\epsilon}^+} \leq a) \leq \frac{1}{\epsilon(1+\delta)}\left( \left(1 + \epsilon\delta\right) -  \left( 1 - \frac{\epsilon\delta}{2}\right)\left(1 - \frac{K_{p,\epsilon}^+ - 1}{n-1}\right) + \frac{\epsilon\delta}{2} \right).
    $$
    Since $\delta,\epsilon \in (0,1)$, we can bound the expression on the right to get
    $$
    \prob(M_{n-K_{p,\epsilon}^+} \leq a) \leq \delta + \frac{1}{2\epsilon}\left(\frac{K_{p,\epsilon}^+ - 1}{n-1}\right).
    $$
\end{proof}

\begin{lemma}
    Let $\epsilon,\eta,\delta \in (0,1)$. There exists an $L \geq 0$ such that, for any $n \geq 0$ and any $p \in (0,\eta)$ such that $K_{p,\epsilon}^- \geq L$, we have,
    $$
        \prob\left(M_{n-K_{p,\epsilon}^-} \leq \binom{K_{p,\epsilon}^-}{2}\right)  \leq \delta + \frac{1}{2\epsilon}\left( \frac{K_{p,\epsilon}^+-1}{n-1} \right).
    $$
\end{lemma}

\begin{proof}
    By the previous lemma, we have some $L \geq 0$ such that, whenever satisfies $K_{p,\epsilon}^+ \geq L$,
    $$
    \prob\left(M_{n-K_{p,\epsilon}^+} \leq \frac{(1-\epsilon)(1-p)K_{p,\epsilon}^+}{p}\right) \leq \delta/2 + \frac{1}{2\epsilon}\left( \frac{K_{p,\epsilon}^+-1}{n-1} \right).
    $$
    Now, consider the event
    $$
    E = \left\{ M_{n-K_{p,\epsilon}^-} \leq \binom{K_{p,\epsilon}^-}{2}\right\} \bigcap \left\{ M_{n-K_{p,\epsilon}^+} \geq \frac{(1-\epsilon)(1-p)K_{p,\epsilon}^+}{p}\right\}.
    $$
    To finish the proof it suffices to show that there is some $L^* \geq 0$ such that, when $K_{p,\epsilon}^- \geq L^*$, we have $\prob(E) \leq \delta/2$. 
    
    If $M_{n-k} + X_{n-k}$ first exceeds $\binom{k}{2}$ at some step $k$ with $K_{p,\epsilon}^- \leq k \leq K_{p,\epsilon}^+$, then the process freezes at this value, and so
    $$
    M_{n-K_{p,\epsilon}^-} \geq \binom{K_{p,\epsilon}^-}{2},
    $$
    implying that $E$ cannot occur. One consequence of this is that we may replace $(X_{n-k} : K_{p,\epsilon}^- \leq k \leq K_{p,\epsilon}^+)$ with a collection of independent (untruncated) geometric random variables $(X_{n-k}^* : K_{p,\epsilon}^- \leq k \leq K_{p,\epsilon}^+)$ without decreasing the probability of $E$ occurring $E$. Note also that $Y_k \leq n-k-2$ deterministically for all $0 \leq k \leq n-2$. Thus, defining $E^*$ to be the event,
    $$
    \left\{ M_{n-K_{p,\epsilon}^+} + \sum_{K_{p,\epsilon}^- + 1}^{K_{p,\epsilon}^+} (X_{n-k}^* - k) \leq \binom{K_{p,\epsilon}^-}{2}\right\} \bigcap \left\{ M_{n-K_{p,\epsilon}^+} \geq \frac{(1-\epsilon)(1-p)K_{p,\epsilon}^+}{p} \right\},
    $$
    we have that $\prob(E) \leq \prob(E^*)$. Next observe that
    $$
    \binom{K_{p,\epsilon}^-}{2} \leq \frac{(1-\epsilon)(1-p)K_{p,\epsilon}^-}{p},
    $$
    and so
    \begin{equation}\label{eq:end_my_suffering_prelude}
        E^* \subseteq \left\{  \sum_{k = K_{p,\epsilon}^- + 1}^{K_{p,\epsilon}^+} (X_k^* - k) \leq \frac{(1-\epsilon)\left(K_{p,\epsilon}^- -K_{p,\epsilon}^+\right)(1-p)}{p} \right\}. 
    \end{equation}
    A quick computation gives that
    $$
        \sum_{k = K_{p,\epsilon}^- + 1}^{K_{p,\epsilon}^+} k = \binom{K_{p,\epsilon}^+}{2} - \binom{K_{p,\epsilon}^-}{2}.
    $$
    Moreover, for any $\beta \in (0,1)$, there exists $L(\beta) \geq 0$ such that
    \begin{equation}\label{eq:end_my_suffering_please}
    \left|\frac{(1-\epsilon)K_{p,\epsilon}^+(1-p)}{p\binom{K_{p,\epsilon}^+}{2}}-1\right| \leq \beta \text{ and } \left|\frac{(1-\epsilon)K_{p,\epsilon}^-(1-p)}{p\binom{K_{p,\epsilon}^-}{2}} -1\right| \leq \beta,
    \end{equation}
    when $K_{p,\epsilon}^- \geq L_2(\beta)$. Hence, for $K_{p,\epsilon}^- \geq L(1/2)$, we have
    \begin{equation}\label{eq:suffering_didnt_end}
    \sum_{k = K_{p,\epsilon}^- + 1}^{K_{p,\epsilon}^+} k + \frac{\left((1-\epsilon)K_{p,\epsilon}^- - (1-\epsilon)K_{p,\epsilon}^+\right)(1-p)}{p} \leq \frac{1}{2}\left(\binom{K_{p,\epsilon}^+}{2} - \binom{K_{p,\epsilon}^-}{2}\right).
    \end{equation}
    Moreover, (\ref{eq:end_my_suffering_please}) also gives, for $K_{p,\epsilon}^- \geq L(3/4)$,
    \begin{equation}\label{eq:lower_bound_expectation_final_lemma}
    \ex\left[ \sum_{k = K_{p,\epsilon}^- + 1}^{K_{p,\epsilon}^+} X_k^* \right] \geq \frac{3}{4}\left( \frac{1}{1-\epsilon}\binom{K_{p,\epsilon}^+}{2} - \frac{1}{1-\epsilon}\binom{K_{p,\epsilon}^-}{2} \right) \geq \frac{3}{4}\left(\binom{K_{p,\epsilon}^+}{2} - \binom{K_{p,\epsilon}^-}{2} \right).
    \end{equation}
    Using (\ref{eq:end_my_suffering_prelude}) and (\ref{eq:suffering_didnt_end}) we have
    $$
    \prob(E) \leq \prob\left( \sum_{k = K_{p,\epsilon}^- + 1}^{K_{p,\epsilon}^+} X_k^* \leq \frac{1}{2}\left(\binom{K_{p,\epsilon}^+}{2} - \binom{K_{p,\epsilon}^-}{2}\right) \right).
    $$
    Finally, applying the negative binomial bounds in Lemma \ref{lem:manybounds} along with (\ref{eq:lower_bound_expectation_final_lemma}) to the right side of the above inequality we obtain a constant $C(\eta) > 0$ such that, for $K_{p,\epsilon}^- \geq L(1/2) \vee L(3/4)$,
    \begin{equation}\label{eq:end_my_suffering_v3}
    \prob\left( E \right) \leq \exp\left(-C(\eta)\left(\binom{K_{p,\epsilon}^+}{2} - \binom{K_{p,\epsilon}^-}{2}\right)\right).
    \end{equation}
    By making $K_{p,\epsilon}^-$ large, the upper bound in (\ref{eq:end_my_suffering_v3}) can be made arbitrarily small, and so less than $\delta/2$, which completes the proof.
\end{proof}

\section{Questions and future research}\label{sec:openquestions}

The work here answers some of the basic questions one may have about the Kingman coalescent on $G_{n,p}$, however, there are still many questions worthy of study. 
One direction for future work could be to focus on the Kingman process on general graphs. We have two quite concrete questions concerning this topic.

\begin{enumerate}
    \item Does there exist a sequence of graphs $(G_n)_{n = 1}^\infty$, with $G_n$ a graph on $n$ vertices, and a function $f:\N \to \N$ with $f(n)/\log(n) \to \infty$ as $n \to \infty$ such that, with probability tending to $1$ as $n \to \infty$, $\height(F(G_n))) \geq f(n)$?
    \item Let $G = ([n] , E)$ be a graph, and let $H = ([n],E')$ be such that $E' \subseteq E$. Does the number of components in $F(H)$ stochastically dominate $F(G)$? Does $\height(F(G))$ stochastically dominate $\height(F(H))$? 
\end{enumerate}

We note that answering (ii) in the affirmative would also answer (i) in the negative, as it would imply that $\height(F(K_n))$ is the largest possible height that can be attained from the Kingman process on any graph on the vertex set $[n]$.

A second direction could be to study the Kingman coalescent on other random graphs. A crucial part of our analysis here was the fact that vertex degrees in $G_{n,p}$ are all identically distributed and exchangeable. What if our underlying graph did not have this property? There are many examples of inhomogeneous random graphs one could choose, but for a concrete example, consider a random graph $G_n$ with some fixed degree sequence $(d_1,...,d_n)$. What can be said about the structure of $\king(G_n)$? When the degree sequence is not regular (i.e., we do not have $d_1 = \ldots = d_n$), how do the statistics of the height of a given vertex in $F(G_n)$ depend on its degree. We expect that high degree vertices are typically found at larger heights than low degree vertices.

It would also be interesting to study other coalescent rules on random graphs. As was mentioned in Section \ref{sec:background}, Kingman's coalescent is not the only commonly studied coalescent process. Deriving results similar to those of this article when we consider a generalization of the additive coalescent rather than Kingman's would be interesting. Moreover, the inhomogeneous random graph case should be just as interesting for the additive coalescent as it is for Kingman's coalescent. Some results concerning the multiplicative coalescent on $G_{n,p}$ are already known, but there is much yet to be done \cite{addario2017scaling,addario2021geometry}.

\section*{Acknowledgements}

CA acknowledges the support of the NSERC Postgraduate Scholarship-Doctoral. LAB acknowledges the support of the NSERC Discovery Grants program and the Canada Research Chairs program.

\setstretch{1.0}
\bibliographystyle{alpha}
\bibliography{ref}

\end{document}